\documentclass[english,11pt]{article}
\usepackage{authblk}
\usepackage[T1]{fontenc}
\usepackage[latin9]{inputenc}
\usepackage{amscd}
\usepackage{amsthm,amsmath,amsfonts,amssymb}
\usepackage{verbatim}
\usepackage{lmodern}
\usepackage{graphicx}
\usepackage{enumerate}
\usepackage{color}
\makeatletter
\numberwithin{equation}{section}
\numberwithin{figure}{section}
\theoremstyle{plain}
\newtheorem{thm}{\protect\theoremname}
  \theoremstyle{plain}
  \newtheorem{lemma}[thm]{\protect\lemmaname}
    \newtheorem{prop}[thm]{\protect\propname}

\newtheorem*{thm*}{Theorem}

\numberwithin{thm}{section}

\newtheorem{cor}[thm]{Corollary}

\theoremstyle{remark}
\newtheorem*{rem}{Remark}
\makeatother

\usepackage{babel}
\providecommand{\propname}{Proposition}

\providecommand{\lemmaname}{Lemma}
\providecommand{\theoremname}{Theorem}

\newcommand{\imag}{\operatorname{Im} \,}
\newcommand{\real}{\operatorname{Re} \,}
\renewcommand{\Im}{\imag}
\renewcommand{\Re}{\real}

\newcommand{\ee}{\epsilon}

\newcommand{\ZZ}{\mathbb{Z}}

\newcommand{\E}{\mathbf{E}}
\newcommand{\PP}{\mathbf{P}}
\newcommand{\Prob}{\mathbf{P}}
\newcommand{\smallsquare}{\scalebox{0.6}{$\square$}}

\newcommand{\hcap}{\operatorname{hcap}}
\newcommand{\diam}{\operatorname{diam}}

\newcommand{\dist}{\operatorname{dist}}
\newcommand{\hm}{\textrm{hm}}

\newcommand{\ball}{\mathcal{B}}

\newcommand{\Z}{{\mathbb Z}}

\newcommand{\R}{{\mathbb{R}}}
\newcommand{\C}{{\mathbb C}}
\newcommand{\rad}{{\rm rad}}

\newcommand{\SLE}{\text{\tiny SLE}}
\newcommand {\G} {{\mathcal G}}

\newcommand{\LERW}{\text{\tiny LERW}}

\def \Half {{\mathbb H}}

\def \F {{\cal F}}
\newcommand {{\wind}} {{\rm wind}}
\newcommand {{\Sine}}{S}
\let \setminus \smallsetminus
\let \le \leqslant
\let \leq \leqslant
\let \ge \geqslant
\let \geq \geqslant
\let \epsilon \varepsilon
\let \phi \varphi

\newcommand{\whoknows}  {{\mathcal A}}

\newcommand{\paths} {{\mathcal K}}
\newcommand{\saws}{\mathcal{W}}
\newcommand {\eset}{{\emptyset}}
\newcommand {\Square} {{\mathcal S}}

\AtEndDocument{\bigskip{\footnotesize%
  (Lawler) \textsc{Department of Mathematics, University of Chicago} \par  
  \textit{E-mail address}: \texttt{lawler@math.uchicago.edu} \par
  
  \addvspace{\medskipamount}
  (Viklund) \textsc{Department of Matematics, KTH Royal Institute of Technology} \par
  \textit{E-mail address}: \texttt{fredrik.viklund@math.kth.se}
}}

\title{The Loewner difference equation and convergence of loop-erased random walk}
\author{Gregory F. Lawler}
\affil{University of Chicago}

\author{Fredrik Viklund}
\affil{KTH Royal Institute of Technology}

\begin{document}
\maketitle
\begin{abstract}
We revisit the convergence of loop-erased random walk, LERW, to SLE$_2$ when the curves are parametrized by capacity. We construct a Markovian coupling of the chordal version of LERW and chordal SLE$_2$ based on the Green's function for LERW as martingale observable and using an elementary discrete-time Loewner ``difference'' equation. This coupling is different than the ones previously considered in this context.   Our recent work on the convergence of LERW parametrized by length to SLE$_2$ parameterized by Minkowski content uses specific features of the coupling constructed here. 
\end{abstract}
\section{Introduction, set-up, and main results}
\subsection{Introduction}
Loop-erased random walk (LERW) is the random self-avoiding path one gets after erasing the loops in the order they form from a simple random walk. In the plane, which is the only case we consider here, it was proved in \cite{LSW04} that LERW has a conformally invariant lattice size scaling limit, namely SLE$_2$. In this paper we revisit this in the case of chordal LERW, proving the result in a slightly different framework than \cite{LSW04}. We need the theorem in this form for our proof of convergence of LERW parametrized by length to SLE$_{2}$ parametrized by $5/4$-dimensional Minkowski content \cite{LV_LERW_natural}. In order to describe our results we will first discuss the work in \cite{LSW04,LV_LERW_natural} and then elaborate on the results of this paper.

The proof in \cite{LSW04} is based on a description of LERW viewed as a continuous curve in terms of Loewner's differential equation, see e.g. \cite[Chapter 4]{Lbook}. In the case of SLE$_{\kappa}$ the Loewner driving process is $\sqrt{\kappa}$ times a standard Brownian motion. The main step is to show that the LERW driving process converges to Brownian motion with variance parameter $2$. The way this is done is by first identifying a \emph{martingale observable}. This is a lattice function which for a fixed lattice point is approximately a martingale with respect to the LERW. One needs to be able to approximate the observable well in rough domains by some continuum quantity with conformal symmetries. In \cite{LSW04} a discrete Poisson kernel was used as observable, converging in the scaling limit to a conformally invariant version of the usual Poisson kernel. The martingale property translates via the Loewner equation to an approximate martingale property of the Loewner process. The argument  produces an estimate on the variance of the increments and from this information one can couple with Brownian motion using Skorokhod embedding.

 Our proof follows the same basic idea but is based on a different observable: the LERW Green's function, that is, the probability that the LERW passes through a given vertex inside the domain. (Since LERW is a self-avoiding walk this probability is also equal to the expected number of visits to the vertex, hence the terminology.) By the domain Markov property the Green's function evaluated at a fixed vertex is a LERW martingale. The approximation result, which is also important for \cite{LV_LERW_natural}, was proved in \cite{BLV}. More precisely, that paper proves that the LERW Green's function properly renormalized converges  with a polynomial convergence rate in the scaling limit to the SLE$_{2}$ Green's function, which is conformally covariant and explicitly known. The theorem does not need assumptions on boundary regularity. Recall that the SLE$_\kappa$ Green's function is the limit as $\ee \to 0$ of the renormalized probability that an SLE$_\kappa$ curve gets within distance $\ee$ of a given point inside the domain. The observable used in \cite{LSW04} is specific to LERW but the Green's function is not. Many of the estimates given here apply to other models as well, assuming one has established convergence to the appropriate SLE$_{\kappa}$ Green's function with sufficient control of error terms. However, such a convergence result is presently known only for LERW.

LERW is a random self-avoiding walk on a lattice (we use $\ZZ^2$) and as such can be viewed either as a continuous curve traced edge by edge or as a sequence of Jordan domains obtained by removing the faces touched (and disconnected from the target point) when walking along the LERW. These viewpoints are of course essentially equivalent but other considerations may make one more convenient than the other. For example, Jordan domains cane be easier to work with analytically. In this paper we adopt the second point of view. We exploit a fundamental robustness of Loewner's equation: the analysis is based on a difference version of Loewner's equation which uses only mesoscopic scale information about the growth process. The difference equation does not require the conformal maps to come from a curve, only that the sequence of maps is generated by composing maps corresponding to small hulls of controlled diameter and capacity. There is still a discrete ``Loewner process'' representing the growth on a mesoscopic scale, up to a uniform multiplicative error. (But this process does not uniquely determine the evolution.) The resulting argument is in a sense more elementary. We explain how to compare solutions to the difference equation corresponding to nearby Loewner processes, and write down formulas for some of the usual important processes such as the derivative and conformal radius. 
 
The difference equation also produces a coupling of the Loewner chains. From this one obtains quantitative estimates comparing the solutions but more work is needed to compare the actual growth processes. At this stage it is more natural to work with curves again. We work out the estimates along the lines of \cite{JV} (which discusses the radial case) utilizing a simple bottleneck estimate for chordal LERW given in \cite{LV_LERW_natural}.

The coupling constructed in this paper, while similar to previous such couplings such as the one in \cite{LSW04}, differs from them in  two important ways. The first is the use of square domains and the second is that it is Markovian for the coupled pair of LERW and SLE driving functions. The former is convenient for using convergence results such as \cite{BLV} while the latter is crucial for the argument in \cite{LV_LERW_natural}.

We have tried to provide a reasonable amount of detail and to make the paper fairly self-contained with the hope that it will be read not only by experts but also as an introduction to these techniques.  

\subsection{Discrete quantities}
We now discuss the discrete quantities we will use. We want the setup to match exactly that of \cite{LV_LERW_natural}, so in this section there will necessarily be some overlap in the presentation.
\begin{itemize}
\item Let $A$ be a finite subset of $\Z^2$, and write $\partial_e A$ for
the edge boundary of $A$, that is, the set of edges  of
$\Z^2$ with exactly one endpoint in $A$.  We specify
elements of $\partial_e A$ by $a$, the midpoint of the edge; this is unique up to the orientation. We  write\
$a_-,a_+$ for the endpoints of the edge in $\Z^2\setminus A$ and $A$, respectively.
Note   that 
\[  a_-, b_- \in \partial A:= \{z \in \Z^2 \setminus A: \dist(z, A) = 1\}, \]
\[     a_+, b_+ \in  \partial_iA := \{z \in A: \dist(z, \partial A) = 1\}. \]
 We also write  
$e_a = [a_-,a_+], e_b = [b_-,b_+]$ for the edges oriented from the outside to the
inside.

\item{Let $\whoknows$
denote the set of triples $(A, a, b)$ where
  $A$  is  a finite, simply connected subset of $\Z^2$  
containing the origin,  and $  a,   b$ are 
elements
of $\partial_e A$ with $a_- \neq b_-$.
  We allow $a_+ = b_+$. Sometimes we slightly abuse notation and write $A \in \whoknows$ when $A$ is a simply connected subset of $\Z^2$  
containing the origin.}

\item  let $\Square= \{x+iy \in \C: |x|,|y| \leq 1/2\}$
be the closed square of side length one centered at the origin and
$\Square_z = z + \Square$. If $(A,a,b) \in \whoknows$, let
 $D_A$ be the corresponding simply connected domain defined
 as the interior of
 \[               \bigcup_{z \in A} \Square_z . \]
  This is a simply connected Jordan domain whose boundary is a subset of the edge set of the dual graph of $ \Z^2$. Note that $a,b \in \partial D_A$. We refer to $D_A$ as a ``union of squares'' domain, slightly abusing terminology.

\item{Let
$F = F_{A,a,b}$ denote a conformal map from
$D_A$ onto $\Half$ with $F( a ) = 0, F( b ) = \infty$.  This
map is defined only up to a dilation; later we will fix
a particular choice of $F$.
Note that $F$ and $F^{-1}$
  extend continuously to the boundary of the domain (with the appropriate
  definition of continuity at infinity). 
   }
   
\item  {For $z \in D_A$, we define the important conformal invariants
\[  \theta_{A,a,b}(z) = \arg F(z), \quad \Sine_{A,a,b}(z) = \sin \theta_{A,a,b}(z), 
\]
which are
independent of the choice of $F$, since $F$ is unique up to scaling.  
Also for $z \in \Half$, we write 
\[   \Sine(z) =  \sin[\arg(z)].\]
Note that $(\arg z) / \pi$ is the harmonic measure in $\Half$ of the negative real line and   $\sin [\arg z]$ is comparable to 
the minimum of the harmonic measures of the positive and negative real lines.}

\item We write $r_A(z)=r_{D_A}(z)$ for the conformal radius of $D_A$
with respect to $z$.  This is usually defined for any simply connected domain $D$ as $r_{D}(z)=\phi'(z)^{-1}$ where  $\phi: D \to \mathbb{D}$ is the Riemann map with $\phi(z)=0, \phi'(z)>0$. We can also compute it from $F$ by
\[        r_A(z) = 2\, \frac{\Im F(z) }{|F'(z)|}, \]
which is independent of the choice of $F$.
\item{Let $(A,a,b) \in \mathcal{A}$. If a confomal transformation $F: D_{A}\to \Half, F(a) = 0, F(b)=\infty$ as above has been fixed we can consider half-plane capacity with respect to $F$ as follows. Let $K \subset D_{A}$ be a compact set such that $D_{A}\setminus K$ is simply connected. The half-plane capacity of $K$ (with respect to $F$) is defined by the usual half-plane capacity of $F(K)$ in $\Half$, see Section~\ref{sect:deterministic}. In this context it is also convenient to define $R_{F}=R_{A,a,b,F}=4|(F^{-1})'(2i)|$ which is the conformal radius of $D_A$ seen from $F^{-1}(2i)$. }
 \item{Suppose $D$ is an analytic simply connected domain containing $0$ as an interior point. Let $N >1$. We sometimes want to consider a lattice approximation of $D$ with mesh $N^{-1}$, and we define it as follows. We take $A=A(N,D) \in \mathcal{A}$ to be the largest discrete simply connected set such that $D_{A} \subset N \cdot D$. We write $$\check{D} = N^{-1}D_{A}$$ for the scaled domain. Then $\check{D}$ is a simply connected Jordan domain which approximates $D$ from the inside and converges to $D$ in the Carath\'eodory sense (with respect to $0$) as $N \to \infty$. If $a, b \in \partial_e A$ are given, we write $\check a, \check b \in \partial \check D$ for $N^{-1}a, N^{-1}b$, respectively. }

\item{A \emph{walk} $\omega = [ \omega_0, \ldots, \omega_n]$ is a sequence of nearest neighbors in $\Z^2$. The length $|\omega| = n$ is by definition the number of traversed edges.}
\item{If $A \in \whoknows$ and $z,w \in A$, we write 
  $\paths_A(z,w)$ for the
  set of walks $\omega$  starting at $z$, ending
  at $w$, and otherwise staying in $A$. }
  \item{  
  The simple random walk measure  $p$ assigns to each walk measure $p(\omega) = 4^{-|\omega|}$. The two-variable function \[G_A(z,w):=p\left(\paths_A(z,w) \right)\] is the simple random
  walk \emph{Green's function}.}
  \item{
  If $a,b \in \partial_e A$, there is an obvious bijection between
  $\paths_A(a_+,b_+)$ and $\paths_A(a,b)$,  the set of 
  walks starting with edge $e_a$, ending with 
  $e_b^R$ and otherwise staying in $A$.
    Here we write $\omega^R$ for the
  reversal of the path $\omega$, that is, if
  $\omega = [\omega_0,\omega_1,\ldots,\omega_k]$, then
  $\omega^R = [\omega_k,\omega_{k-1},\ldots,\omega_0]$.
We sometimes write $\omega: a \rightarrow b$ for walks in $\paths_A(a,b)$ with the condition to stay in $A$ implicit.  } 
  
 \item{We
  write $H_{\partial A}(a,b)$ for the total random walk measure of  $\paths_A(a,b)$.
It is easy to see that 
$   H_{\partial A}(a,b) =   G_A(a_+,b_+)/16$ (this is sometimes called a last-exit decomposition). 
 The factor of $1/16 = (1/4)^2$ comes from the $p$-measure of the edges $e_a,e_b$. $H_{\partial A}(a,b)$ is called the \emph{boundary Poisson kernel}.}
 
\item{A \emph{self-avoiding walk} (SAW) is a walk visiting each point at most once. We write $\saws_A(z,w) \subset \mathcal{K}_A(z,w)$ for the set of SAWs from $z$ to $w$ staying in $A$.
  We will write $\omega$ for
general walks and reserve $\eta$ for SAWs.  We   write $\saws_A(a,b)$ similarly when $a,b$ are boundary edges.  }

 \item{The loop-erasing procedure takes a walk as input and outputs a SAW, the \emph{loop-erasure} of $\omega$. Given a walk $\omega = [\omega_0, \ldots, \omega_n]$, we define its loop-erasure $\text{LE}[\omega]=[\text{LE}[\omega]_0, \ldots, \text{LE}[\omega]_k]$ as follows.
 \begin{itemize}
 \item{If $\omega$ is self-avoiding, set $\text{LE}[\omega] = \omega$.}
 \item{Otherwise, define $s_0 = \max\{ j \le n: \omega_j = \omega_0\}$ and let $\text{LE}[\omega]_0 =\omega_{s_0}$. }
 \item{For $i \ge 0$, if $s_i < n$, define $s_{i+1} = \max\{ j \le n: \, \omega_j = \omega_{s_i+1}\}$ and set $\text{LE}[\omega]_{i+1} = \omega_{s_i+1}
  = \omega_{s_{i+1}}$.}
 \end{itemize}
Note that if $e_a \oplus \omega
\oplus e_b^R \in \paths_A(a,b)$, then
LE$[e_a \oplus \omega
\oplus e_b^R] =  e_a \oplus \text{LE}[\omega]
\oplus e_b^R $.
 }
 \item  Given a measure on walks, the loop-erasing procedure induces a natural measure on SAWs. We define $\hat P_{A,a,b}$, the ``loop-erased'' $p$-measure, on $\saws_A(a,b)$ by
 \[     \hat P_{A,a,b}(\eta) = \sum_{\omega \in \paths_A(a,b): \; \text{LE}(
 \omega) = \eta}
p(\omega). \]
This can also be written
\begin{equation}\label{LERW_loops}
\hat{P}_{A,a,b}(\eta) = p(\eta)\Lambda_A(\eta),
\end{equation}
where $m(\eta;A)=\log \Lambda_A(\eta)$ is the loop-measure (using $p$) of loops intersecting $\eta$ and staying in $A$, see, e.g., \cite[Section 2]{BLV}. 
This does not define a probability measure; indeed the total mass  $\hat P_{A,a,b}[\saws_A(z,w)] = H_{\partial A}
(a,b)$.  Let \[\Prob_{A,a,b} = \frac{\hat P_{A,a,b}}{
H_{\partial A}
(a,b)}\] denote the probability measure obtained by normalization. This is the probability law of (chordal) \emph{loop-erased random walk} (LERW) in $A$ from $a$ to $b$.
\end{itemize}

With these definitions in place, we can state the main result from \cite{BLV}, which we will make significant use of in this paper.  
We emphasize that no assumptions about the discrete domain $A$ are made.  
\begin{lemma}  There exists $\hat c > 0$ and $u  >0$ such that the following holds. Suppose $(A,a,b) \in \whoknows$ and that $\zeta \in A$ is such that $S_{A,a,b}(\zeta) \ge r_A(\zeta)^{-u}$,  then
\begin{equation}
 \label{BLV1}   \Prob_{A,a,b}\{ \zeta \in \eta\} = \hat c \, \, r_A(\zeta)^{-3/4}\Sine_{A,a,b}^3(\zeta)
 \, \left[1 +O\left(r_A(\zeta)^{-u}\Sine_{A,a,b}^{-1}(\zeta) \right)\right].
 \end{equation}
 \end{lemma}
 	
 We have not estimated $u$ except $u>0$. For the rest of the paper we fix a value of $u$ such that \eqref{BLV1} 
holds.
We can also write \eqref{BLV1} using the
 SLE$_2$ Green's function for $(D_A,a,b)$ which is further discussed in Section~\ref{sect:sle}. Let
 \[    G_{D_A}(\zeta;a,b) = \tilde c \,  r_A(\zeta)^{-3/4}
 \, S_{A,a,b}^3(\zeta), \]
for a specific (but unknown) constant
$\tilde c >0$ that will be defined later.  We may  rewrite \eqref{BLV1}
 as
\begin{equation}  \label{BLV2}
   \Prob_{A,a,b}\{ \zeta \in \eta\} =  c_* \,
     G_{D_A} (\zeta;a,b) \,\left[1
       + O\left(r_A(\zeta)^{-u} \right) \;  \Sine_{A,a,b}^{-1}(\zeta)
       \right], 
       \end{equation}
  where $c_* =   \hat c/\tilde c$ is a positive constant whose exact value is presently unknown.

\subsection{Continuum quantities}\label{sect:sle}
Recall that chordal SLE$_\kappa$ in $\Half$ is a random continuous curve $\gamma(t), t \ge 0,$ constructed by first solving the Loewner differential equation 
\[
\partial_t g_t(z) = \frac{2/\kappa}{g_t(z) - B_t}, \quad g_0(z) = z \in \Half.
\]
Here $B_t$ is standard Brownian motion.  We shall only consider $0 < \kappa < 8$ in this paper, and  primarily $\kappa=2$. The conformal maps $g_t(z)$ can be expanded at infinity as
\[
g_t(z) = z + \frac{(2/\kappa) t}{z} + O(|z|^{-2}).
\]
Then for each $t \ge 0$ we define the SLE$_\kappa$ curve and trace by
\[
\gamma(t) = \lim_{y \to 0+} g^{-1}_t(U_t + iy), \quad \gamma_t:= \gamma[0,t].
\]
This limit is known to almost surely exist for each $t$ and to define a continuous curve $t \mapsto \gamma(t)$ in $\Half$ growing from $0$ to $\infty$
(see Section~\ref{sect:deterministic} for more information on the Loewner equation). This defines SLE$_\kappa$ in the reference domain $\Half$ with marked boundary points $0,\infty$ and we extend the definition to any simply connected domain $D$ with two marked boundary points (prime ends) $a,b$ (we write $(D,a,b)$ for such a triple) by transferring the curve by a Riemann map taking $\Half$ to $D$, $0$ to $a$, and $\infty$ to $b$. Using Brownian scaling, one can see that this is well defined if one allows for a linear time reparametrization.

The Green's function for SLE$_\kappa$ in a domain $(D,a,b)$ is defined by \[G_D(z,a,b) = 
\lim_{\ee \to 0} \ee^{d-2}\Prob \left\{\dist(z, \gamma_\infty) \le\ee \right\}=  \tilde{c}\, r_D^{d-2}(z)S_{D,a,b}^{\beta}(z),\]
where $\gamma$ is chordal SLE$_\kappa$ in $D$ from $a$ to $b$, 
\[d= 1+ \frac \kappa 8, \qquad \beta = \frac 8\kappa-1\]  
is the dimension of the SLE$_\kappa$ trace, and the SLE$_{\kappa}$ boundary exponent, respectively, and $\tilde{c} \in (0,\infty)$ is a constant whose exact value is not known. Here $r_D$ and $S_{D,a,b}$ are defined in the same manner as for the union of squares domains $D_A$ discussed in the previous subsection.

\subsection{Main results}
Here we state the main result of this paper in a form which we use in \cite{LV_LERW_natural}.  See also Proposition~\ref{prop:main-coupling} for the coupling of the Loewner processes and Lemma~\ref{lem:coupling-of-maps} for the coupling of the Loewner chains. These are important steps on the way to the main result and they are also directly used in \cite{LV_LERW_natural}.
Given parametrized continuous curves taking values in $\mathbb{C}$, $\gamma^1(t), t\in [s_1, t_1],$ and  $\gamma^2(t), t \in [s_2,t_2]$,  we measure their distance using a metric $\rho$ defined by
\[
\rho(\gamma_1,\gamma_2) = \inf_\alpha \left[ \sup_{s_1 \le t \le t_2}|\alpha(t) - t| + \sup_{s_1 \le t \le t_1}|\gamma^2(\alpha(t)) - \gamma^1(t)| \right],
\]
where the supremum is taken over increasing homeomorphisms $\alpha: [s_1,t_1] \to [s_2,t_2]$.
\begin{thm}
There exists $p_0 > 0$ and for each $p \in (p_0,1]$ a $q > 0$ such that the following holds. Suppose $(D,a',b')$ is given, where $D$ is an analytic simply connected domain containing $0$ and $a',b' \in \partial D$ are distinct boundary points. Then there exists $N_0=N_0(D,a',b',p) < \infty$ such that  the following holds. 

For each $N$, let $(A,a,b) \in \mathcal{A}$ be chosen as above so that $D_A$ approximates $N \cdot D$ and $a,b$ are chosen among the edges in $\partial_e A$ nearest $N\cdot a', N\cdot b'$, respectively.  

Let $\eta$ be LERW in $A$ from $a$ to $b$ and let $\check{\eta}(t) = N^{-1}\eta(t), t \in [0,1],$ be the continuous curve in $\check{D}$ from $\check a$ to $\check b$ obtained by parametrizing $\eta(t)$ by half-plane capacity with respect to $F:D_{A} \to \Half$ (taking $a$ to $0$ and $b$ to $\infty$). Suppose $F$ satisfies $R=R_{A,a,b,F}  \ge N^p$ whenever $N \ge N_0$. 

Then for each $N \ge N_0$ there is a coupling of $\check{\eta}$ and a chordal SLE$_2$ path $\check{\gamma}(t), t \in [0, t_{\check{\gamma}}],$ in $\check{D}$ from $\check{a}$ to $\check{b}$ (parametrized in the same way as $\check \eta$) for which
\[
\Prob \left\{\rho(\check{\eta}, \check{\gamma}) > R^{-q}  \right\}  < R^{-q}.
\]

\end{thm}
Let us make a few remarks.
\begin{itemize}
\item{Choosing $p < 1$ corresponds to measuring capacity using a map normalized at a point which gets closer to $\partial \check{D}$ as $N \to \infty$, which allows to consider a larger and larger portion of the curve as $N \to \infty$. In particular, this implies convergence with a polynomial rate of the paths stopped at a suitable mesoscopic distance from $\check{b}$.}
\item{It is not difficult to show that the SLE$_{2}$ in $\check{D}$ from $\check{a}$ to $\check{b}$ is close to an SLE$_{2}$ in $D$ from $a$ to $b$ in the metric $\rho$. We will not discuss this in detail here but see Section~7 of \cite{LV_LERW_natural}. 
}
\item{It is not necessary to assume that the domain $D$ is analytic away from $a,b$ in order to deduce convergence, but convergence rates may depend on the specific regularity of the domain.}
\end{itemize}

\subsection*{Acknowledgments}
Lawler was supported by National Science Foundation grant DMS-1513036. Viklund was supported by the Knut and Alice Wallenberg Foundation, the Swedish Research Council, the Gustafsson Foundation, and National Science Foundation grant DMS-1308476. We also wish to thank the Isaac Newton Institute for Mathematical Sciences and Institut Mittag-Leffler where part of this work was carried out.
\section{Discrete and continuous time Loewner chains}\label{sect:deterministic}

The Loewner differential equation is a continuous limit of a Loewner difference
estimates.  The difference estimates hold for sets more general than curves, and 
since we are dealing with ``union of squares'' domains, we will use the difference
formulation.  Here we will review the basics from \cite[Section 3.4]{Lbook} and then
we will give some extensions.  It is important for us to be careful with the error
terms.

We recall that a set 
 $K \subset \Half$ is a {\em (compact $\Half$-) hull}, 
 if $K$ is bounded and 
  $H_K := \Half \setminus
K$ is a simply connected domain.  Let $h_K = \hcap(K)$,
the (half-plane) capacity  which can be defined in two 
equivalent ways:
\begin{itemize}
\item If $B_t$ is a complex Brownian motion and $\tau = \inf\{t:
B_t \in \R \cup K\}$, then 
\[   h_{K}= \lim_{y \rightarrow \infty} y\, \E^{iy}\left[\Im[B_\tau]\right]. \]
\item  If $g_K: H_K \rightarrow \Half$ is the unique conformal
transformation with $g_K(z) = z + o_K(1)$ as $z \rightarrow \infty$,
then
\[    g_K(z) = z + \frac{h_K}{z} + O_K(|z|^{-2}), \;\;\;\;
   z \rightarrow \infty. \]
\end{itemize}
We write the error terms as $o_K, O_K$ to emphasize that they depend
on $K$; the error terms we write below will be uniform over all $K$.
We recall that $h_K \leq r_K^2$ where
\[    r_K = \rad(K) = \sup\{|z|: z \in K\}.\]
There is no lower bound for $h_K$ in terms of radius only; however, there
exists $c < \infty$ such that
\[    h_K  \geq r_K \cdot\max\{\Im(z): z \in K\}. \] 
Let  \[   \Upsilon_{K}(z) = \frac{\Im[g_K(z)]}{|g_{K}'(z)|}, \]
and recall that $2\Upsilon_K(z)$ is the conformal radius of
$H_K$ seen from $z$.   The basic Loewner estimate \cite[Proposition 3.46]
{Lbook}  is
\begin{equation}\label{oct9.1}
    g_K(z) = z + \frac{h_K}{z} + O\left(\frac{r_K \, h_K}
  {|z|^2 }\right) , \;\;\;\;   |z| \geq 2 \, r_K.  
\end{equation}
By applying the Cauchy integral formula
to  $f_K(z) = g_K(z) - z-(h_K/z)$, we
see that
\begin{equation}\label{oct9.1d}
    g_K'(z) = 1 - \frac{h_K^2}{z} + O\left(\frac{r_K \, h_K}
  {|z|^3 }\right) , \;\;\;\;   |z| \geq 2 \, r_K.  
\end{equation}

This is the starting point for the next lemma.  
 \begin{lemma}   \label{oct12.lemma1} There exists
  $c < \infty$ such that the following holds.
Suppose $U \in \R$; $K$ is
a hull with $r_{K} < 1/2$;
 $z=x+iy$; and let $g,r,h,\Upsilon$
 denote \[g_{K + U },\, r_{K}, \, h_K = h_{K + U }, \text{ and }
 \Upsilon_{K + U},\] respectively.  Then $\Im[g(z)] \leq
 y$ and $\Upsilon(z) \leq y$.  Moreover, if 
 $\delta =r^{1/4}, \, h \le \delta r$ and $y \geq \delta$, then
 \[ \left|g(z) - z - \frac{h}{z-U } 
 \right| \leq  c h\delta^2,\]
 \begin{equation} \label{oct9.2}
  \left|g'(z) - 1 + \frac{h}{(z-U )^2} 
 \right| \leq c  h\delta,
 \end{equation}
\[ 
     \left| \Im[g(z)] -y  \, \left[1  -
   \frac{h}{|z-U |^2} \right]\right|
    \leq c y h \delta,\]
 \[ \left|\Upsilon(z) -y
 \left[   1 - \frac{2h\sin^2\theta}{|z-U |^2}
 \right] \right| \leq c yh\delta,\]
In particular, if $\sin \theta \geq \nu$, then
\begin{equation}  \label{oct12.1}
 \frac{\Upsilon(z) }y\leq \left(\frac{\Im(g(z))}y\right)
  ^{2 \nu^2} \, \left[1+O(h\delta) \right]. 
  \end{equation}
 \end{lemma}

\begin{proof} Since $g_{K+U }(z)  
 = g_K(z-U ) + U ,$  it suffices to prove
 the result when $U=0$ which we will assume from now on. 
 The first two inequalities follow immediately from \eqref{oct9.1}
 and \eqref{oct9.1d}, respectively. 
Taking imaginary parts  in the first inequality and using
$|z| \geq \Im(y) \geq \delta$, we get
\begin{align*}
  \Im[g(z)] & =   y \, \left[1 - \frac{h}{|z|^2}
  \right]
  +  O\left(h\delta^2\right) \\
  & =   y \, \left[1 - \frac{h \,
  (\cos^2\theta + \sin^2 \theta)}{|z|^2}
  \right]
  +  O\left(h\delta^2\right)
  \end{align*}
and since $y \geq \delta$ we get the third inequality. 
Since
\[  \left|1 - \frac{h}{z^2} \right| =  1-\Re\left[
\frac{h}{z^2} \right] + O\left(\frac{h^2}{|z|^4} \right)
 = 1+
\frac{h \, (\sin^2 \theta -\cos^2\theta)}{|z|^2}  + O\left(\frac{h^2}{|z|^4} \right),\]\
and $h/|z| \leq r$, we get 
\[  \left|g_K'(z) \right|^{-1}  =
  1+
\frac{h \, (\cos^2 \theta - \sin^2 \theta)}{|z|^2}  + O\left(\frac{hr}{|z|^3} \right),\]
Combining, we get
\[  \Upsilon_K(z) = y \, \left[  1-
\frac{2h \, \sin^2 \theta}{|z|^2}  +  O\left(\frac{h\delta^2}{y}\right)
\right].\]
 \end{proof}
 
 Suppose now we have a sequence of hulls of small capacity
 $K_1,K_2,\ldots$ and locations $U_1,U_2,\ldots  \in \R$ determining a ``Loewner process'', so that, roughly speaking $K_j + U_j$ is near $U_j$.
 Let \[r_j = r_{K_j}, \quad  h_j = h_{K_j}, \quad g^j
  = g_{K_j + U_j}\] and let
  \[  g_j =  {g^j \circ \cdots
   \circ g^1}.\]
 If $z \in \Half$, we define
 \[   z_j = x_j + i y_j =  g_j(z) . \]
 This is defined up to the first $j$ such
 that $z_j-U_j \in K_j$. (Recall that $K_j$ is located near $0$.) 
A key fact (and the basis of 
  the Loewner differential equation) is  
  that the left-hand side of \eqref{oct9.1}
  depends only on $h,U$ and not on the exact shape
  of $K$.  This implies that if we have two sequences
  for which the capacity increments and Loewner processes, $h_j$ and $U_j$, are close, then we would expect the functions $\phi_n$
  to be close for points which are away from the real line.  We give a precise formulation of this
  in the next proposition.  To illustrate the idea, let us sketch a continuum argument first. Suppose $U_{t}, \tilde{U}_{t}$ are continuous, real-valued function, defined on $[0,T]$, and write $\ee:=\sup|U_{t} - \tilde{U}_{t}|$. Write $g_{t}, \tilde{g}_{t}$ for the corresponding Loewner chains and $z_{t}=g_{t}(z) -U_{t}$ and $\tilde{z}_{t}= \tilde{g}_{t}(z) - \tilde{U}_{t}$. Suppose that $\delta \le  \min\{\Im z_{T}, \Im \tilde{z}_{T}\}$. If $G_{t} = g_{t}(z) - \tilde{g}_{t}(z)$, then 
  \[
  \dot{G}_{t} = \psi_{t}[-G_{t} + (U_{t} - \tilde{U}_{t})], \quad G_{0}=0, \quad \text{where } \psi_{t} = \frac{a}{z_{t}\tilde{z}_{t}}.
  \]
  By solving the ODE and using the definition of $\ee$ we have
  \[
  |G_{t}| =\left|\int_{0}^{t}e^{-\int_{s}^{t} \psi_{r} dr}\psi_{s} (U_{s} - \tilde{U}_{s}) ds \right| \le \ee \int_{0}^{t}e^{\int_{s}^{t} |\psi_{r}| dr}|\psi_{s} | ds.
  \]
  From here we integrate and then proceed by applying Cauchy-Schwarz' inequality: if $y = \Im z$, then
  \[
  \left(\int_{0}^{t} |\psi_{r}| dr \right)^{2} \le \int_{0}^{t}\frac{a}{|z_{r}|^{2}}dr  \int_{0}^{t}\frac{a}{|\tilde z_{r}|^{2}} dr = \log \frac{\Im z}{ \Im z_{t}} \log \frac{\Im  z}{ \Im \tilde z_{t}} \le \left( \log (\delta/y) \right)^{2}.
  \]
  The identity comes from taking the imaginary part of the Loewner equation and the last estimate uses the definition of $\delta$. Hence we get the estimate
  \[
  |g_{t}(z) - \tilde{g}_{t}(z)| = |G_{t}| \le  c \, (\ee/\delta) \, (y \wedge 1). 
  \]
  It is possible to estimate in terms of other norms relating $U_{t}$ and $\tilde{U}_{t}$ and, as we will see, continuity is not necessary to assume.

 \begin{prop}\label{prop:loewner-comparison}  There exists $1 < c < \infty$ such
 that the following holds.  Suppose $(K_1,U_1),
 $ $ (K_2,U_2)\ldots$ and $(\tilde K_1,\tilde U_1),
 (\tilde K_2,\tilde U_2),$ $\ldots$ are two sequences as above
 with corresponding $r_j, h_j, g^j, g_j$ and
 $\tilde r_j,  \tilde h_j, \tilde g^j, \tilde g_j$.
 Let \[0 <   h <
 r^2 < \epsilon^2 < \delta^8 < 1/c,\] 
 and $n  \leq 1/h$ and suppose that
 for all $j=1,\ldots,n$,
 \[      |h_j -  h| \leq  hr/\delta  , \;\;\;\;
 |\tilde h_j - h| \leq hr/ \delta , \]
  \[      r_j, \tilde r_j  \leq r  , \]
   \[     |U_j -\tilde U_j| \leq \epsilon. \]
Suppose $z = x+iy \in \Half$ and let
$z_n = x_n +iy_n = g_n(z), \tilde z_n =
\tilde x_n + i \tilde y_n = \tilde g_n(z).$
Then, if $y_n,\tilde y_n \geq \delta$,
\begin{equation}  \label{halloween.1}
  |g_n(z) - \tilde g_n(z)| \leq c \, (\epsilon/\delta) \,
     (y \wedge 1).
     \end{equation}
 Moreover, if we assume that $y_n \geq 2 \delta$ and make
 no a priori assumptions on $\tilde y_n$, then $\tilde y_n \geq \delta$ holds,
 and hence \eqref{halloween.1} follows in this case, too.
\end{prop}

\begin{proof} Note that $nh \leq 1  $, and hence if $y \geq 3$,
we know that $y_n \geq \delta$.  Without loss of generality, we will assume
that $y \leq 3$; for $y \geq 3$, we can use the fact that
$\phi_n - \tilde \phi_n$ is a  bounded
holomorphic function on $\{\Im(w) > 3\}$ that goes to zero as $w \rightarrow \infty$,
and hence
\[       |g_n(z) - \tilde g_n(z)| \leq \max\{|g_{n}(s+3i)
   - \tilde g_n(s+3i)|: s \in \R\} . \]

 Using Lemma \ref{oct12.lemma1}, and that $r < \delta^4$, we see
that for $j=0,\ldots,n-1$, 
 \begin{equation}  \label{nov8.4}
    z_{j+1} = z_{j} + \frac{h}{z_j-U_j}
   \ + O\left( {h\delta^2} 
    \right) , 
    \end{equation}
   \[  y_{j+1} = y_{j} \, \left[1 -  \frac {h}{|z_j - U_j|^2}
   +  O\left(h \delta 
    \right)\right] , \]
    and similarly for $\tilde z_j,\tilde y_j$.

 Hence
 \[ y_n  =   y \, \prod_{j=0}^{n-1}
    \left[1 -  \frac{h}{|z_j - U_j|^2} +  O\left( h \delta 
    \right)\right]  = 
     y \, [1+O(\delta)] 
    \,\exp\left\{ -\sum_{j=0}^{n-1}  \frac{h}{|z_j - U_j|^2}\right\}.
\]
 Since $y_n \geq \delta$ and $y \leq 3$, it follows that
\begin{equation}  \label{oct11.2}
 \sum_{j=0}^{n-1}\frac{h}{|z_j - U_j|^2}
     \leq  \log(y/\delta) + O(\delta) , 
     \end{equation}
      and similarly for $(\tilde z_j, \tilde U_j)$.
Using the Cauchy-Schwarz inequality, we see that,
\begin{align}  \label{oct11.3}
\sum_{j=0}^{n-1}\frac{h}{|z_j - U_j|
 \, |\tilde z_j - \tilde U_j|}
 & \leq \left[ \sum_{j=0}^{n-1}\frac{h}{|z_j - U_j|^2}
 \right]^{1/2} \,  \left[ \sum_{j=0}^{n-1}\frac{h}{|
 \tilde z_j - \tilde U_j|^2}
 \right]^{1/2} \nonumber \\
 &  \leq \log(y/\delta) + O(\delta).
   \end{align}
 Let $\Delta_j = z_j - \tilde z_j$. Let us first assume that $|\Delta_j| \leq \delta/2$. By subtracting
 the expressions in \eqref{nov8.4} for $z_j$
 and $\tilde z_j$, we see that 
 \[  \Delta_{j+1} = \Delta_{j}
      + \frac{h \,( U_j -  \tilde{U}_{j} -
       \Delta_j)}{(z_j - U_j) \, (\tilde z_j 
         - \tilde{U}_j)} + O\left(h\delta^2 \right).\]
  This implies that there exists $c$ such that 
  \[  |\Delta_{j+1}| \leq |\Delta_j| \, \left[1
   + \rho_j
   \right] +  c \, \epsilon \, \rho_j
,
  \]
 where 
 \[     \rho_j = \frac{h}{|z_j-U_j| \, |\tilde z_j - \tilde U_j|}
   .\]   Integrating we get,
   \[   |\Delta_{j+1}| \leq c \, \epsilon  \sum_{l=1}^{j} \left( \rho_l
     \, \prod_{k=l+1}^j ( 1 + \rho_k)  \right)\leq c \epsilon (y/\delta).\] 
     The last inequality uses \eqref{oct11.3} and the identity
   \[   1+\sum_{l=1}^n \left( p_l \prod_{k=l+1}^n (1+p_k) \right)
        = \prod_{l=1}^n (1+p_l).\]
   
Hence we see that
 \[   |\Delta_n| \leq c \, \epsilon \, ( y/\delta), \]
 provided that the right-hand side is less than $\delta/2$. 
 Since $y \leq 3$ and $\epsilon \leq \delta^4$,
 this will be true if $\delta$ is sufficiently small.
 
 For the final assertion, suppose that $j$ is such that $\tilde{y}_j \ge \delta$. Then since $\ee \le \delta^4$, we can use \eqref{halloween.1} to see that $|y_j - \tilde{y}_j| \le c (\ee/\delta)y \le O(\delta^4)$. Since $y_j \ge 2\delta$, it follows that $\tilde{y}_j \ge 2\delta(1-O(\delta^3))$. But $|\tilde{y}_{j+1} - \tilde{y}_j| \le c' h_j/y_j \le  O(\delta^7)$. Consequently, as long as $\delta$ is sufficiently small, taking $c$ larger if necessary, we can continue until $j=n$.

\end{proof}

\begin{cor}
Suppose we make the assumptions of the previous proposition, but
replace the condition $y_n \geq 2\delta $
with
\[         \Upsilon_n(z), \tilde \Upsilon_n(z)
 \geq   2 (2\delta)^{2 \nu^2}
    ,\]
where 
 \[  \nu = \min_{0 \le j \le n}\left\{\sin \left[\arg \left(g_j(z) - U_j \right) \right] \right\} . \]
Then the results still hold for $\delta$ sufficiently small.
\end{cor}

\begin{proof} Using \eqref{oct12.1}, we see that for
  $\delta$ sufficiently small
\[         \Upsilon_n(z), \tilde\Upsilon_n(z) \leq   2 y_n^{2 \nu^2}
    .\]
\end{proof}
The next proposition, which is important for \cite{LV_LERW_natural}, gives a familiar representation of the derivative of the uniformizing map and a related geometric estimate.
\begin{prop}
\label{prop:deriv-lb}  There exists $1 < c < \infty$ such
 that the following holds.  Suppose $(K_1,U_1),
 $ $ (K_2,U_2)\ldots$ is a sequence as above
 with corresponding $r_j, h_j, g^j, g_j$.
 Let \[0 <   h <
 r^2 <  \delta^8 < 1/c,\] 
 and $n  \leq 1/h$ and suppose that
 for all $j=1,\ldots,n$,
 \[      |h_j -  h| \leq  hr/\delta, \quad r_j \le r.  \]
Suppose $z = x+iy \in \Half$ and let
$z_n = x_n +iy_n = g_n(z)$.
Then if $y_n \geq \delta$,
\begin{equation}  \label{halloween.2}
  |g_n'(z)| = \exp\left\{-\sum_{j=0}^{n-1} \Re \frac{h}{(z_{j} - U_{j})^2}\right\} \left(1+O(\delta) \right).  
     \end{equation}
     In particular, there is a constant $c$ such that if
\begin{equation}\label{jan26.1}
 \nu=\min_{0\le j \le n} \left\{\sin\left[ \arg\left(g_j(z) - U_j \right) \right] \right\},
 \end{equation}
 then,
 \begin{equation}\label{jan26.2}
 |g'_n(z)| \ge c \left(\frac{y_n}{y}\right)^{1-2\nu^2}.
 \end{equation}
\end{prop}
\begin{proof}
By the chain rule and Lemma~\ref{oct12.lemma1} we have
\begin{align*}
\log|g'_n(z)| & = \sum_{j=1}^{n}\log|(g^j)'(z_{j-1})| \\
& = \sum_{j=0}^{n-1} \log\left|1-\frac{h}{(z_{j} - U_{j})^2} + O(h\delta)\right|\\
&  = -\sum_{j=0}^{n-1} \left(\Re \frac{h}{(z_j-U_j)^2} + O(h\delta)\right).
\end{align*}
This proves the first claim. For the second assertion, note that \eqref{jan26.1} implies
\[
-\Re \frac{h}{(z_j - U_j)^2} = -\left(1-2 S_j^2 \right) \frac{h}{|z_j-U_j|^2} \ge -\left(1-2\nu^2 \right) \frac{h}{|z_j-U_j|^2}, 
\]
where 
\[
S_j = \sin\left[ \arg(g_j(z) - U_j \right]. 
\]
But in the proof of Proposition~\ref{prop:loewner-comparison} we saw that
\[
\exp\left\{ -\sum_{j=0}^n\frac{h}{|z_j-U_j|^2} \right\} = (y_n/y)\left(1+ O(\delta) \right).
\]
Combining these estimates finishes the proof.\end{proof}
\subsection{Reverse-time Loewner chain}
In this section we consider a reverse-time version of the discrete Loewner chains. The estimates are completely analogous to the forward-time case discussed above
and indeed could be concluded almost directly from them, so we will omit proofs and only state the needed results.

We associate with a hull $K$ a conformal map,
\[
f_K : \Half \to H_{K}, \quad f_K(z) = z - \frac{h_K}{z} + o(|z|^{-1}),
\]
and of course, $f_K = g_K^{-1}$.
 \begin{lemma}   \label{oct12.lemma1.rev} There exists
  $c < \infty$ such that the following holds.
Suppose $U \in \R$; $K$ is
a hull with $r_{K} < 1/2$;
 $z=x+iy$; and write $f,r,h,\Upsilon$
 for $f_{K + U },r_{K},h_K = h_{K + U }$ respectively.  Then $\Im \left[f(z) \right] \ge
 y$.  Moreover, if 
 $\delta = r^{1/4}$ and $y \geq \delta$, then
 \[ \left|f(z) - z + \frac{h}{z-U } 
 \right| \leq  c h\delta^2,\]
 \[  \left|f'(z) - 1 - \frac{h}{(z-U )^2} 
 \right| \leq c  h\delta,\]
 \begin{equation}
    \label{nov8.1}
     \left| \Im \left[f(z) \right] -y  \, \left[1  +
   \frac{h}{|z-U |^2} \right]\right|
    \leq c y h \delta,\end{equation}
 \end{lemma}
 
 {
  We will consider sequences $(K_j, U_j)$, where the $K_j$ are centered hulls as above and $U_j \in \mathbb{R}$ are the locations of the hulls.
 Let \[r_j = r_{K_j}, \quad  h_j = h_{K_j}, \quad f^j
  = f_{K_j + U_j}.\] Also let
  \[  f_j =  {f^1 \circ \cdots
   \circ f^j}.\]
   and notice that
   \[
   (f_j)^{-1} = {g^j \circ \cdots
   \circ g^1}.
   \]
   We see that the situation here in a sense is more symmetric than in the continuum time case: There one has to consider a time-reversed driving term in order to use the reverse flow to represent the inverse of the uniformizing map. The actual inverse map satisfies a partial differential equation involving a $\partial_z$-derivative, and this is one of the reasons why it is convenient to work with the reverse-flow instead of the actual inverse maps.
   
 If $z \in \Half$, we define
 \[   z_j = x_j + i y_j =  f_j(z) . \]
 This is defined for all positive $j$.  
 \begin{prop}\label{prop:reverse-time-comparison}  There exists $1 < c < \infty$ such
 that the following holds.  Suppose $(K_1,U_1),
 $ $ (K_2,U_2)\ldots$ and $(\tilde K_1,\tilde U_1),
 (\tilde K_2,\tilde U_2),$ $\ldots$ are two sequences as above
 with corresponding $r_j, h_j, f^j, f_j$ and
 $\tilde r_j,  \tilde h_j, \tilde f^j, \tilde f_j$.
 Let \[0 <   h <
 r^2 < \epsilon^2 < \delta^8 < 1/c,\] 
 and $n  \leq 1/h$ and suppose that
 for all $j=1,\ldots,n$,
 \[      |h_j -  h| \leq  hr/\delta  , \;\;\;\;
 |\tilde h_j - h| \leq hr/ \delta , \]
  \[      r_j, \tilde r_j  \leq r  , \]
   \[     |U_j -\tilde U_j| \leq \epsilon. \]
Suppose $z = x+iy \in \Half$ and let
$z_n = x_n +iy_n = f_n(z), \tilde z_n =
\tilde x_n + i \tilde y_n = \tilde f_n(z).$
Then, if $y \geq \delta$, and $y_n, \, \tilde y_n \le \delta$,
\begin{equation}  \label{halloween.12}
  \left|f_n(z) - \tilde f_n(z) \right| \leq c \, (\epsilon/y) \,
     (\delta \wedge 1)
     \end{equation}
     and
     \[
     \left|y|f'_n(z)| - y|\tilde{f}_n'(z)| \right| \le c \, (\epsilon/y) \,
     (\delta \wedge 1).
     \]
\end{prop}
\begin{proof}
The last estimate follows from \eqref{halloween.12} using the Cauchy integral formula.
\end{proof}
\begin{prop}
\label{prop:deriv-ub}  There exists $1 < c < \infty$ such
 that the following holds.  Suppose $(K_1,U_1),
 $ $ (K_2,U_2)\ldots$ is a sequence as above
 with corresponding $r_j, h_j, f^j, f_j$.
 Let \[0 <   h <
 r^2 <  \delta^8 < 1/c,\] 
 and $n  \leq 1/h$ and suppose that
 for all $j=1,\ldots,n$,
 \[      |h_j -  h| \leq  hr/\delta, \quad r_j \le r.  \]
Suppose $z = x+iy \in \Half$ and let
$z_n = f_n(z)$.
Then if $y \geq \delta$,
\begin{equation}  \label{halloween.2inv}
  \left|f_n'(z) \right| = \exp\left\{\sum_{j=0}^{n-1} \Re \frac{h}{(z_{j}-U_j)^2}\right\} \left(1+O(\delta) \right).  
     \end{equation}
     In particular, there is a constant $c$ such that if
\begin{equation}\label{jan26.1inv}
 \nu=\min_{0\le j \le n} \left\{\sin\left[ \arg\left(z_j - U_j \right) \right] \right\},
 \end{equation}
 then,
 \begin{equation}\label{jan26.12}
 |f'_n(z)| \le c \left(\frac{y_n}{y}\right)^{1-2\nu^2}.
 \end{equation}
\end{prop}

 \begin{rem}
If an estimate such as \eqref{jan26.12} is known for one of the Loewner chains, the worst-case blow-up $y^{-1}$ in the estimate \eqref{halloween.12} can be improved. This is not needed here so we will not give details, but see \cite{JVRW} for continuous time versions.
 \end{rem}
 
\subsection{Expansion of the SLE Green's function} 
We consider now the SLE$_\kappa$ Green's function which in the case $\kappa =2$ equals
\[
G_{D}(z,a,b) :=\tilde{c} \, r_D^{-3/4}(z)S_{D,a,b}^3(z).
\]
We shall later use the LERW analog as an
observable to help prove convergence to SLE$_2$. For this, we need to understand how the scaling limit, that is, the SLE Green's function, changes if the domain is perturbed  by growing a small hull. The computation is no more difficult for general $\kappa$ so we will not assume $\kappa=2$ here.

 Let $z_\pm = i \pm 1$. Then
 \[   \sin[\arg(z_\pm)] = \frac {\sqrt 2}{2}.
\]  

\begin{lemma}  \label{lemma:taylor}
Suppose $K$ is a hull, $r=r_K=\diam(K), h= h_K=\hcap(K)$, 
 $U \in [-(rh)^{1/3},(rh)^{1/3}]$,  $z_\pm = i \pm 1$.
Then,
\[  \Im[g(z_\pm)] = 1 - \frac{h}{2} + O(hr) , \]
\[   |g'(z_\pm)| = 1 + O(hr) , \]
\[  \sin \left[\arg(g(z_\pm) - U) \right]
  = \frac{\sqrt 2}{2} \, \left[1 \pm \frac{U}{2} + \frac{ U^2}{8}
     - \frac h 2 + O(hr + r^3)\right].\]
\end{lemma}

\begin{proof}  We will
show the result for $z_+$; the argument for $z_-$
is identical.  Let us write
\[  w= g(z_+) = x + iy =   |w| \, e^{i\arg w},\]
where $\arg w \in [0, \pi]$.
Using \eqref{oct9.1}, 
\[  x = 1 + \frac h2 + O(hr) , \;\;\;\;\; y = 1 -
\frac h2 + O(hr) , \;\;\;
   |w| = \sqrt 2 + O(hr) . \]
   Moreover,
\[  \sin \arg w = \frac{y}{|w|} = \frac 1 {\sqrt 2}
   - \frac{h}{2 \sqrt 2} + O(hr),\;\;\;
 \;  \arg w = \frac {\pi}{4} - \frac h2 + O(hr).\]
 Using
\eqref{oct9.2} and the fact that $z_\pm^2$ is purely
imaginary, we have 
\[      |g'(z_\pm)|
  = 1 + O(hr).\]
  We now want to expand $\arg(g(z_+) - U) = \arg(w-U)$ up to $O(hr) +  O(r^3)$. Proceeding directly by Taylor expansion becomes a bit involved, so we will first exploit the harmonicity. For the moment, let us assume that $U\geq 0$.
Let $\psi(\zeta) = \arg(\zeta-U) - \arg(\zeta)$. By the maximum principle $\psi(\zeta)$ equals $\pi$ times the probability that a Brownian motion exits $\mathbb{H}$ in $[0,U]$. Since $\psi$
is a positive harmonic function, and $|z_+ - w| = O(h)$,
we have
\[        \left|\psi(z_+) - \psi(w) \right|
      \leq   c h \left|\psi(z_+) \right| = O(hr), \]
 that is, $\psi(w) = \psi(z_+)\,[1 + O(hr)]$.
Hence, using the Poisson kernel for $\mathbb{H}$,
  \begin{align*}
  \arg \left(w-U \right) & = \psi(w) + \arg(w) \\
  & = \psi(z_+) + \frac{\pi}{4}-\frac h2 + O(hr) \\
  & = \int_0^U \frac{dt}{(1-t)^2+1}+ \frac{\pi}{4}-\frac h2 + O(hr)\\
  & = \frac{\pi}{4}-\frac h2+ \frac{U}{2} + \frac{U^2}{4} +O(hr) + O(r^3).
  \end{align*}
If $U < 0$, we need to consider the probability of hitting the boundary in
$[U,0]$, but the same basic argument shows that in this case
\begin{align*}
\arg \left(w-U \right) & = \arg(w) - \int_{U}^0 \frac{dt}{(1-t)^2+1} \\
& = \frac{\pi}{4} - \frac{h}{2} + \frac{U}{2} + \frac{U^2}{4} + O(hr)+ O(r^3).
\end{align*}
Doing the analogous computation with $z=z_-$ we get
 \begin{eqnarray*} \arg(g(z_\pm) - U) =\frac{\pi}{4}  - \frac{h}{2}
   + \frac{U}{2}  \pm \frac{U^2}{4} + O(hr)+O(r^3).
   \end{eqnarray*}
Finally we use the elementary formulas
 \[   \sin\left(\frac \pi 4 + \epsilon\right)
 =  \sin (\pi/4) \, \left[1+\epsilon - \frac {\epsilon^2}{2}
  + O(\epsilon^3)\right], \]
  and
  \[   \sin\left(\frac {3\pi}{4} + \epsilon\right)
 =  \sin (3\pi/4) \, \left[1-\epsilon - \frac {\epsilon^2}{2}
  + O(\epsilon^3)\right]. \]
 We conclude
 \[   \sin \left(\arg(g(z_\pm) - U) \right)= \frac{\sqrt 2}{2 }\left[1 \pm \frac{U}{2} + \frac{U^2}{8}
     - \frac h 2 + O(hr)+ O(r^3) \right].\]

 \end{proof}

The expansion of the observable is an immediately consequence. 
We will use this result only with $\kappa = 2$,
but we state it so that it can be applied to other discrete models converging
to SLE$_\kappa$ for $0 < \kappa < 8$ if the analog of \eqref{BLV1} is known.  
 \begin{prop}\label{prop:obs-exp}Suppose we are in the setting of Lemma~\ref{lemma:taylor}.
 If $0 < \kappa < 8$ and
 \[ 
  \alpha = \frac \kappa 8 - 1 , \;\;\;\; \beta =
  \frac 8 \kappa - 1 , \]
then 
\begin{multline}
     \Upsilon(z_\pm)
      ^\alpha \,\sin^\beta \left( \arg(g(z_\pm) - U) \right) \\ 
 \label{cor5}
    = 
 \left(\frac{\sqrt 2}{2}\right)^\beta \, \left(  1 \pm A_{\kappa}
   \, U +  B_{\kappa}
        \,  \left[U^2 -  \frac {h\kappa}{2} \right]  + O_\kappa(hr + r^3)\right),
 \end{multline}
 where
\[
A_{\kappa}=\frac{4}{\kappa} - \frac 12, \quad B_{\kappa}= \frac{8}{\kappa^2} - \frac {2} \kappa + \frac 18, \quad \Upsilon(z_\pm)=\frac{ \Im[g(z_\pm)]}{  |g'(z_\pm)|}.
\]
 
 \end{prop}

\section{Coupling the Loewner processes and Loewner chains}\label{sect:coupling}
In  this section we derive the basic coupling results relating the Loewner processes and the corresponding Loewner chains. The method we follow is the same as in \cite{LSW04} but we work with a different observable, namely the LERW Green's function, and with the discrete Loewner equation. In order to be able to use the results in \cite{LV_LERW_natural} we also need to be more careful with measurability properties.  The resulting coupling is different from the one of \cite{LSW04}.
We will give quantitative estimates (in terms however of the unknown exponent $u$ chosen so that \eqref{BLV1} holds), but we have not bothered to optimize exponents. 

\subsection{Loewner process}
We start with $(A,a,b) 
\in \whoknows$, so that $A$ is a lattice domain with marked boundary edges $a,b$. Recall that we write $F:D_A \rightarrow \Half$ for a conformal transformation
with $F(a) = 0, F(b) = \infty$.  As we have noted before, there is a one-parameter
family of such transformations $F$, so we will now fix one of them. We define \[R = R_{A,a,b,F} = 4|(F^{-1})'(2i)|\] and note that $R$ equals the conformal radius of $D_A$ seen from $F^{-1}(2 i)$. (Of course, the choice of the point $2i$ from whose preimage in $D_{A}$ we measure the size of the domain is quite abritrary -- we want to grow curved up to capacity $1$, which in $\Half$ have maximum imaginary part $\sqrt{2}$.)  We will prove facts
for $(A,a,b, F)$ with $R$ sufficiently large and we will not always be explicit about this.

Fix a mesoscopic scale $h$, defined by
\[
h = R^{-2u/3},
\]
where $u$ is the exponent from \eqref{BLV1}. This is somewhat arbitrary, but we will use that $R^{-u} = O(h^{6/5})$. 

Before going into detailed estimates, let us pause here and give an overview of the argument. We first grow a piece of LERW of capacity $h$; more precisely,
 we will stop it the first time its image in $\Half$ has capacity $h$ or reaches diameter $h^{2/5}$. (In \cite{LSW04} the analogous stopping time is defined slightly differently, in terms of the capacity increment and the driving term displacement.) But we shall prove that with very large probability the latter event does not occur. Indeed, since LERW is unlikely to ``creep'' along the boundary we expect the diameter of the increment in $\Half$ to be of order $h^{1/2}$. So, we have a mesoscopic piece $\eta_h$ of LERW of capacity (very near) $h$ growing from $a$ in $A$. The domain Markov property of LERW implies that for $\zeta$ sufficiently far away from $\eta_h$,
\[
p(\zeta) = \E \left[ \E\left[ p(\zeta) \mid \eta_h\right] \right] = \E \left[p_h(\zeta) \right],
\] 
where $p(\zeta) = \Prob_{A,a,b} \left\{ \zeta \in \eta \right\}$ and $p_h(\zeta) = \Prob_{A',a',b} \left\{ \zeta \in \eta \right\}$ is computed in the smaller domain $A'$ with the LERW piece $\eta_h$ removed and with marked edges the ``tip'' of $\eta_h$ and $b$. Using \eqref{BLV1} we can express both sides of the equation in terms of the SLE$_2$ Green's function for $D_{A}$ and $D_{A'}$ (both Jordan domains), and using Proposition~\ref{prop:obs-exp} we can can expand $p_h(\zeta)$ in terms of the Loewner process displacement $\xi$. By doing this for two different choices of $\zeta$ we get two independent equations which allow us to show that $\E[\xi] = 0$  and $\E[\xi^2-\hcap[\eta_h]] = 0$ up to a very small error of $O(h^{6/5})$. These are the two critical estimates.

This argument can then be iterated thanks to the domain Markov property. We do so enough times to build a macroscopic piece of LERW. The outputs are uniform estimates on the conditional expectations and conditional variances of the Loewner process displacements in the sense of a sequence of $\Half$-hull increments and positions, exactly as in Section~\ref{sect:deterministic}. The position displacements nearly form a discrete martingale (with a controlled error), and can, with some work, be coupled with Brownian motion using Skorokhod embedding. From the estimate on the variance of the displacement, we conclude that it is a standard Brownian motion.  
\subsubsection{One step}
We begin by discussing the estimates for one mesoscopic increment of the LERW.
Suppose $\eta$ is a SAW chosen from the LERW probability measure $\Prob_{A,a,b}$ and that $A_j = A \setminus \eta^j$, where $\eta^j=\eta[0,j]$ is considered taking (microscopic) lattice steps. We introduce a stopping time $m=m_1$  depending on $A,a,b,F$ as follows:
\begin{equation}\label{eq:m}
m=\min\left\{j \ge 0: \,    \hcap\left[K_{j} \right] 
\geq h \text{ or }   \diam\left[K_j \right]  \geq h^{2/5} \right\},\end{equation}
where $K_j= F(D \setminus D_{A_j})$ is the image in $\Half$ of the LERW hull. (That is, the squares touched by the LERW together with those squares that are disconneced from $b$.) 
We define for $j = 0,1, \ldots,$ 
\[
t_j = \hcap[K_j], \quad r_j = \diam[K_j]. 
\]
Using the Beurling estimate, we have the easy bounds
\[         
t_m \leq h + O(R^{-1}) , \quad r_m \le h^{2/5} + O(R^{-1/2}).
  \]
  Given the definition of $m$ we expect however that $t_m$ is very close to $h$ and that $r_m$ is in fact very close to $h^{1/2}$.
\begin{lemma}\label{jan26.lemma1} There exist $0 < \alpha, c < \infty$
such that for $R$ sufficiently large, if $(A,a,b,F)$
are as above, then for $K > 0$, 
\[  \Prob_{A,a,b}\left\{ r_m \geq K \, h^{1/2}  \right\}  \leq 
  c \, e^{-\alpha K}. \]
 \end{lemma}

 \begin{proof}  We sketch the proof here; for details
 see  Section \ref{newproofsec}. 
 We consider $m'$, the first $j$ such that 
  $\hcap[K_j] \geq h$ or  $\diam[K_j]  \geq 
 4 \sqrt h $.     The key step is to show
 that there exists uniform $\rho > 0$ such that with probability at
 least $\rho$, we have  $\diam[K_{m'}] <  4 \sqrt h$.  
 If this happens we stop; otherwise,
 we do the same thing on the new walk.  The probability of doing this
 $J$ times without success is at most $(1-\rho)^{J}$.   If we have
 succeded within $J$ steps than $\diam[K_m] \leq O(J \sqrt h)$.\end{proof}
Note that $F(D_{A_m})$ is an unbounded simply connected subset of $\mathbb{H}$ and let $g$ be the uniformizing conformal map normalized so that
  \[
  g: F(D_{A_m}) \to \mathbb{H}, \quad g(z) = z + o(z), \quad z \to \infty.
  \] 
  We write
  \[
  \xi = g(a_m)
  \]
  and finally set $F_m = g \circ F$.

  \begin{lemma}\label{lem:nov19.1}There exist  $
  0 < \beta,c < \infty$ such for $N$ sufficiently large and $(A,a,b,F)$ as above,
   \[    \left|\E_{A,a,b}\left[\xi\right] \right| \le c h^{6/5}, \;\;\;\;\;
    \left|\E_{A,a,b}\left[\xi^2 - h\right] \right| \le c h^{6/5},\]
    and
\begin{equation}  \label{feb17.1}
 \E_{A,a,b}\left[\exp\left\{\beta \,\xi \, h^{-1/2}\right\} \right] \leq c .
 \end{equation}
  \end{lemma}
\begin{proof}
Write $z_\pm = 2(i \pm 1)$ and $H=F^{-1}$. Then $H$ maps $\mathbb{H}$ onto $D_A$. Let $w,\zeta_+,\zeta_-$ be
points in $A\subset \Z^2$ closest to $H(2 i),H(z_+), H(z_-)$,
respectively. In case
of ties, we choose arbitrarily. Note that the domain Markov property for loop-erased random walk
implies that
\begin{equation}\label{eq:nov19.1}  \Prob_{A,a,b}\{\zeta_{\pm} \in \eta\}
  = \E_{A,a,b} \left[\Prob_{A_m,a_m,b}\{
  \zeta_\pm \in \eta \} \right]. 
  \end{equation}
  We will estimate the two sides of this equation. To keep the
  notation simpler we will write $z=z_\pm$ and $\zeta=\zeta_\pm$.  We begin with the left-hand side for which we can use \eqref{BLV1} directly. Recall that $R=4|H'(2i)|$. By distortion estimates we know that
\[    \left|F(w) - 2 i\right| ,
     \left| F(\zeta) - z\right| \leq O(R^{-1})\]
     and
     \[
     |F'(\zeta)|^{-1} = |H'(z)|\left(1+O(R^{-1})\right).
     \]
     Hence,
     \[
     r_{D_A}(\zeta) = 2|H'(z)|\left(1+O(R^{-1}) \right), \quad \sin(\arg F(\zeta)) = \frac{\sqrt{2}}{2} + O(R^{-1}). 
     \]
It follows from \eqref{BLV1} that
\[
\Prob_{A,a,b}\{\zeta \in \eta\}= c_0|H'(z)|^{-3/4}\left(\left(\frac{\sqrt{2}}{2}\right)^3+O(h^{6/5}) \right),
\]     
where we used that $R^{-u} = O(h^{6/5})$ and set $c_0:=\hat{c} \, 2^{-3/4}$. We now estimate the right-hand side of \eqref{eq:nov19.1}. By the chain rule and distortion estimates,
\[
r_{D_{A_m}}(\zeta) = 2 \frac{\Im g(z)}{|g'(z)|} |H'(z)|(1+O(R^{-1})),\]
\[\sin \left(\arg\left[ g\circ F(\zeta) -\xi \right] \right) = \sin \left[ \arg\left( g(z) -\xi \right) \right] +O(R^{-1}). 
\]
So, by \eqref{BLV1} 
\begin{multline}\label{nov19.2}
\Prob_{A_m,a_m,b}\{
  \zeta  \in \eta \}  \\ = c_0 |H'(z)|^{-3/4}\left(\frac{\Im g(z)}{|g'(z)|}\right)^{-3/4}\left(\sin^3\left[\arg(g(z) - \xi)\right]+ O(h^{6/5}) \right)  \\  = 2^{3/2} \Prob_{A,a,b}\{
  \zeta \in \eta \} \left(\frac{\Im g(z)}{|g'(z)|}\right)^{-3/4}\left(\sin^3\left[\arg(g(z) - \xi)\right]+ O(h^{6/5}) \right).
\end{multline}
Note that $r=\diam(K_m) \le h^{2/5} + O(R^{-1})$ so  there is a constant $c$ such that $|\xi| \le ch^{2/5}$ for $h$ sufficiently small. Hence $O(hr + r^3) = O\left(h^{6/5}\right)$ and we can apply Proposition~\ref{lemma:taylor} with $\kappa=2$ to get 
\begin{multline*}
2^{3/2}\left(\frac{\Im g(z)}{|g'(z)|}\right)^{-3/4}\sin^3\left[\arg(g(z) - \xi)\right]  =   1 \pm \frac{3}{2}
   \, \xi +  \frac{9}{8}
        \,  \left(\xi^2 -  t_m \right)  + O(h^{6/5}).
\end{multline*}
Using this, by combining \eqref{eq:nov19.1} with \eqref{nov19.2}, we see that
\[
\E_{A,a,b}\left[ \pm \frac{3}{2}
   \, \xi +  \frac{9}{8}
        \,  \left(\xi^2 -  t_m \right)\right] = O(h^{6/5}).
\]
These equations imply
\[
\left| \E_{A,a,b}\left[\xi\right] \right| = O(h^{6/5}), \quad \left|\E_{A,a,b}\left[\xi^2 - t_m\right] \right|=O(h^{6/5}).
\]
Using Lemma 
\ref{jan26.lemma1} we can conclude  both that $t_m = h +o(h^{6/5})$ and
the final assertion of the lemma.
 
\end{proof}

\begin{prop} \label{prop:skor} There exist $0 < \alpha, C < \infty$
such that  one can define on the same probability space
a random variable $\xi$ with the distribution $\Prob_{A,a,b}$ and
a standard Wiener process $W_t$, and a stopping time $\tau$ for $W_t$
such that $\xi - \mu = W_{\tau}$ where $\mu = \E_{A,a,b}[\xi]$.
Moreover,
\[   \E\left[\tau\right] = \E_{A,a,b}\left[
    (\xi - \mu)^2\right] = h + O(h^{6/5}), \]
and if
\[       W^* = \max\{|W_t|: t \leq \tau\}, \]
then
\[     \E\left[\exp\left\{\alpha W^*h^{-1/2}\right\}\right]
   \leq C . \]

\end{prop}

\begin{proof}  This can be seen using Lemma~\ref{lem:nov19.1} from the standard construction via Skorokhod embedding.
The last inequality uses \eqref{feb17.1}.

\end{proof}

\subsubsection{Sequence of steps} \label{coupling2}
We start with $(A,a,b)$ and $F$ as before, and having chosen a mesoscopic scale $h$. We have defined a step in a sequence of $4$-tuples $(A,a,b, F) \to (A_{m}, a_m, b, F_m)$ which corresponds to a mesoscopic capacity increment of the LERW. This process can be continued to define a sequence of $4$-tuples.  The estimates of Lemma~\ref{lem:nov19.1} will hold as long as the conformal radii (seen from the preimage of $2i$) of the decreasing domains are comparable to that of $A$. By the domain Markov property this corresponds to a sequence of stopping times for the LERW path stopped at mesoscopic capacity increments. 

Let us be more precise. We start with a LERW $\eta$ in $A$ from $a$ to $b$ and $\eta^j=\eta[0,j]$ as usual. We write $D_0 = D_A, D_j = D_{A_j}$. Set $m_0=0, m_1=m$, where $m$ is as in \eqref{eq:m}. Then for $n=1, 2 \dots$, and $j =0, 1, \ldots$, we consider
\[
 K^{n}_j = F_{m_{n-1}}(D_{m_{n-1}} \setminus D_{m_{n-1} + j}), \quad K_j = K^1_j.
\]
Define the stopping times
\[
\Delta_n = \min\left\{ j \ge 0: \, \hcap[K_j^n] \ge h \, \text{ or } \,  \diam[K_j^n] \ge h^{2/5} \right\},
\]
and define $m_n=m_n(h)$ by
\[
m_{n} = m_{n-1} + \Delta_n.  
\]
Write
\[
K^n = K^n_{\Delta_n}.
\]
for the $n$th ``hull increment''.
Then we consider  
   \[t_{m_{n}} =t_{m_{n-1}} + \hcap\left[K^n \right]\]
   so that the squares visited and disconnected by the $n$ first mesoscopic steps of the LERW, $\eta^{m_n}$, has capacity $t_{m_n}$. Also, set 
   \[
  r_{m_{n}} = \diam\left[K^n \right].
   \]
 Let $g^{n}:\Half \setminus K^n \rightarrow \Half$
 be the conformal transformation with $g^{n}(z)
  - z = o(1)$ and set $F_{m_{n}} = g^{n} \circ F_{m_{n-1}}$ and  \[     g_{n} = g^{n} \circ g^{n-1} \circ \cdots
    \circ g^1 . \]
  We also define the ``Loewner process'' 
  \begin{equation}\label{jan12.2} 
   U_{n} = F_{m_n}(a_{m_n}), \end{equation}
 with increments
\[
\xi_n = U_{n}-U_{n-1}.
\]   
We choose the term Loewner process over the more standard ``driving process/term'' since while the SAW determines the $U_{n}$ process, the converse is not true. 
  Write also
  \[
  H_n = F(D_{m_n}) \subset \mathbb{H}.
  \]
 We continue this process until $n_0$,  the first time $n$ such
 that 
 \[    r_{m_n} \geq 3/2 \;\;\;\;\;\; \mbox{ or }
  \;\;\;\;\;\;   t_{m_n} \geq 3/2 . \]
 Note that $n_0 -1 \leq 2/h$ and that for $n < n_0$,
\[  t_{m_n}  < 3/2, \;\;\;\;|U_n| \leq 3/2, \]
  \[   |(F_{m_n}^{-1})'(2 i)| \asymp |(F^{-1})'(2 i)|
  = R/4,\]
Using the Beurling estimate, we can see that
for $n < n_0$, the mesoscopic  increments satisfy
\[        t_{m_n} -t_{m_{n-1}}  \leq h +  O(R^{-1}),\;\;\;\;
       r_{m_n}-r_{m_{n-1}}  	\leq h^{2/5} + O(R^{-1/2}).\]
       With very large probability, after $n_0$ iterations we have built a hull of capacity $3/2 > 1$, as the next lemma shows. Let $\F_n$ denote the $\sigma$-algebra
generated by the LERW domains $(A_0,a_0,b), (A_1, a_1, b), \cdots,
(A_{m_n},a_{m_n},b)$. 
\begin{lemma}\label{lemma:large-cap}
There exist $c, \alpha$ such that\[
\Prob\left\{t_{m_{n_0}}  < 3/2 \right\} \le ch^{-1} e^{-\alpha h^{-1/10}}.
\]
\end{lemma}
\begin{proof}
By Lemma~\ref{jan26.lemma1}  there are constants $\alpha, c$ such that  for $n=1,\ldots, n_0$,
\[
\Prob \left[r_{m_n} \ge h^{2/5} \mid \mathcal{F}_{n-1} \right] \le c e^{-\alpha h^{-1/10}}.
\]
Summing over $n$ gives the lemma.
\end{proof}
 
\begin{lemma}\label{lem:coupling-pt2}
There exists $c < \infty$ such that the following holds. There is $R_0 < \infty$ such that if $R \ge R_0$ and $(A,a,b,F)$ is as above, then there is a coupling of $\eta$ and standard Brownian motion $(W_t, \mathcal{\tilde{F}}_t)$ and 
a sequence of stopping times $\{\tau_n\}$ for $(W_t, \mathcal{\tilde{F}}_t)$ 
such that the following estimates hold:
\begin{enumerate}

\item[$(i.)$]{\begin{equation}\label{jan12.1}
\Prob \left\{\max_{n \le n_0}|\tau_n-nh| > ch^{1/5} \right\} \le c h^{1/5},
\end{equation}}
\item[$(ii.)$]{
\[
\Prob \left\{\max_{n \le n_0}|W_{\tau_n}-U_n| > ch^{1/10} \right\} \le c h^{1/10},
\]}
\item[$(iii.)$]{\[
\Prob \left\{ \max_{n \le n_0} \max_{\tau_{n-1} \leq t \leq \tau_n}
      |W_{t} - W_{\tau_{n-1}}|  > ch^{2/5} \right\} \le c h^{1/10},
\]}
\item[$(iv.)$]{\[
\Prob \left\{ \max_{t \leq \tau_{n_0}}\;\; 
          \max_{t - h^{1/5} \leq s \leq t}\;\;
              |W_t - W_s| > ch^{1/12}\right\} \le c h^{1/10}.
\]}
\end{enumerate}

Moreover, if $\G_{n}$ denotes the $\sigma$-algebra generated by $ \F_{n}$
and $\tilde{\F}_{\tau_{n}}$, then $t \mapsto W_{t + \tau_n}-
W_{\tau_n}$ is independent of $\G_{n}$ and
the distribution of the LERW given
$\G_n$ is the same as the distribution given
$\F_n$.

\end{lemma}
\begin{proof}Using Lemma~\ref{lem:nov19.1} and the domain Markov property we see that there is a constant $c <\infty$ such if $N$ is large enough, on the event $n_0 \ge n$, 
\[ \left| \E\left[\xi_n \mid \F_{n-1}\right] \right|
   \le c h^{6/5}, \]
 \[   \left| \E\left[\xi_n^2 - (t_{m_n}-t_{m_{n-1}}) \mid \F_{n-1}
   \right]\right| \le ch ^{6/5},\]
  \[  \E\left[\xi_n^4\mid \F_{n-1}
   \right]  \le c h^{8/5}.\]
   Here the error terms are uniform in $n \le n_0$.

Let
\[
\delta_n = \xi_n - \E[\xi_n \mid  \F_{n-1}].
\]
This is clearly a martingale difference sequence. We use the Skorokhod embedding theorem (see Proposition~\ref{prop:skor}) to
define a standard Brownian motion $W_t$, generating the filtration  $\tilde \F_t$, and a sequence
of stopping times $0=\tau_0 <\tau_1 < \ldots$ for $W$
such that 
\[    W_{\tau_{n} } - W_{\tau_{n-1}} 
= \delta_n.\]
It is important that this coupling has the property that it does not look ``into the future of the LERW''.  That is to say, 
if $\G_{n}$ denotes the $\sigma$-algebra
generated by $\tilde\F_{\tau_{n}}$ and $\F_{n}$, 
then the Brownian motion $t \mapsto W_{t + \tau_n}-
W_{\tau_n}$ is independent of $\G_{n}$ and
the distribution of the LERW in the future given
$\G_n$ is the same as the distribution given
$\F_n$. 

Since $n_0 = O(h^{-1})$ on the event $n \le n_0$ (which we will assume from now on) we have 
     \[
     \E\left[\sum_{j=1}^n \E[\xi_j \mid  \F_{j-1}]
 \right] = O(h^{1/5}).
     \]
     Hence by the Markov inequality,
 \[  \Prob\left\{\sum_{j=1}^n \E[\xi_j \mid  \F_{j-1}]
  \geq h^{1/10} \right\}
   =O( h^{1/10}).\]
   Therefore, except for an event of probability
 $O(h^{1/10})$, possibly increasing $c$,
 \begin{equation}\label{nov23.2}   |U_n - W_{\tau_n}| \leq c h^{1/10} \;\;\;\;
 \mbox{ for all } n  \le n_0.\end{equation} This gives $(i)$.
We will now compare the capacity increments.
We know that
\[
\E[\delta_n^2-(t_{m_n}-t_{m_{n-1}}) \mid \mathcal{G}_{n-1}] = O(h^{6/5})
\]
and
\[
\E[\delta_n^2-(\tau_n-\tau_{n-1}) \mid \mathcal{G}_{n-1}] = 0.
\]
So we expect that the $\tau$ increments are close to the capacity increments which in turn are deterministic, with very large probability. We will show the first part of this by looking at a suitable martingale.
For this, note that if \[\mu_n = t_{m_n}-t_{m_{n-1}}, \quad \nu_n = \tau_n-\tau_{n-1}, \]
then we have
\[
\E[\mu_n-\nu_n \mid \mathcal{G}_{n-1}] = O(h^{6/5}).
\]
Consider the martingale
\[
M_n = \sum_{k=1}^n Y_k,
\]
where
\[
Y_k = \mu_k - \nu_k - \E[\mu_k-\nu_k \mid \mathcal{G}_{k-1}].
\]
Then 
\[
\E[\mu_n^2 + \nu_n^2 \mid \mathcal{G}_{n-1}] = O(h^{8/5}),
\]
and we can sum this estimate to see that
\[
\E[M_n^2] = \sum_{k=1}^n \E[Y_k^2] = O( h^{3/5}).
\]
Hence by Doob's maximal inequality,
\[    \Prob\left\{\max_{1 \leq k \leq n}
        |M_k| \geq h^{1/5} \right\}
          \leq h^{-2/5} \, \E[M_n^2]
            = O(h^{1/5}).\]
            Since \[\max_{1 \le k \le n}|t_{m_k} - \tau_k| \le  \max_{1 \le k \le n}|M_k| + ch^{1/5},\] we see that  except on an event of probability $O(h^{1/5})$ we have 
            \begin{equation}\label{jan1.1}
             \max_{1 \le k \le n}|t_{m_k} - \tau_k| \le  c h^{1/5}. 
            \end{equation}
By Lemma~\ref{jan26.lemma1} we know that except on an event of probability $o(h^{1/5})$,
\[
\max_{1 \le k \le n}|t_{m_k} - kh| \le  c h^{1/5},
\]
and so we conclude that except on an event of probability $O(h^{1/5})$,
\begin{equation}\label{sle-caps}
\max_{1 \le k \le n}|\tau_k - kh| \le  c h^{1/5}.
\end{equation}
This gives $(ii)$. For $(iii)$ we can use the last estimate of Proposition~\ref{prop:skor} together with Chebyshev's inequality and $(iv)$ follows from a modulus of continuity estimate for Brownian motion.
\end{proof}
 We rephrase the coupling result as follows.
\begin{prop}\label{prop:main-coupling}  There is $c<\infty$ such that the following holds. If $R$ is sufficiently large  we can define a LERW domain configuration
sequence \[ \{(A_j,a_j,b), \;\;\;j =  0,1,\ldots, J\},\] stopping times $m_n, n=0, \ldots, n_0,$ for the LERW, a standard
Brownian motion $W_t, 0 \leq t \leq 1$, and a sequence of increasing
stopping times $\tau_n,  n=0, \ldots, n_0,$
 for the Brownian motion, on the same probability
space such that the following holds.
\begin{itemize}
\item  The distribution of $\{(A_{m_n},a_{m_n},b)\}$ is that
of the LERW domains corresponding to $\Prob_{D_{A},a,b}$ sampled at mesoscopic capacity increments, as described above.
\item  Let $\G_n$ denote the $\sigma$-algebra generated by
$\{(A_j,a_j,b): j=0\ldots,m_n\}$ and $\{W_t: t \leq \tau_n\}$.
Then,
\[       \{(A_j,a_j,b):\,j >m_n \}, \]
\[       \{W_{t+ \tau_n}  - W_{\tau_n}: t\geq 0\}\]
are conditionally independent of $\G_n$ given
$(A_{m_n},a_{m_n},b)$.

\item  There exists a stopping time $n_* \leq n_0$ with respect to
$\{\G_n\}$ such that 
\[   \Prob\{ n_* < n_0\} \leq  c \, h^{1/10}, \]
and such that for $n \leq n_*,$
\[            |W_{\tau_n} - U_n| \leq c \, h^{1/10} ; \]
\[      |\tau_n -nh| \leq c \, h^{1/5};\]
\[      \max_{\tau_{n-1} \leq t \leq \tau_n}
      |W_{t} - W_{\tau_{n-1}}|  \leq c  \, h^{2/5};\]
      \[    \max_{t \leq \tau_n}\;\; 
          \max_{t - h^{1/5} \leq s \leq t}\;\;
              |W_t - W_s| \leq c \, h^{1/12}.\]

 \item  Let $K^{n}$ be the $n$:th mesoscopic hull increment in $\Half$.
 For $n \leq n_*,$ $\hcap \left[ K^{n} \right] \leq h + h^{2}$.
 Moreover, for $n <n_*$, $\hcap \left[ K^{n} \right] \geq h $.

\end{itemize}
\end{prop} 
  \begin{proof}[Proof of Proposition~\ref{prop:main-coupling}]Let $c$ be as in Lemma~\ref{lem:coupling-pt2}. We define $n_*$ to be the minimum of $n_0$ and the first $n$ such that either of
  \[            |W_{\tau_n} - U_n| > c h^{1/10} ; \]
\[      |\tau_n -nh| > c h^{1/5};\]
\[      \max_{\tau_{n-1} \leq t \leq \tau_n}
      |W_{t} - W_{\tau_{n-1}}|  > c h^{2/5};\]
      \[    \max_{t \leq \tau_n}\;\; 
          \max_{t - h^{1/5} \leq s \leq t}\;\;
              |W_t - W_s| > c h^{1/12},\]
              \[
              \hcap\left(K^n \right) < h
              \]
              occurs. Note that if $\hcap\left(K^n \right) < h$, then $\diam(K^n) \ge h^{2/5}$. Hence using  Lemma~\ref{lem:coupling-pt2} and Lemma~\ref{jan26.lemma1} we see that $\Prob\left\{n_* < n_0 \right\} = O(h^{1/10})$.
  \end{proof}
  
   \subsection{Loewner chains}
Given the Brownian motion $W_t$ of Proposition~\ref{prop:main-coupling}, there is a corresponding SLE$_2$ Loewner chain $(g_t^\SLE)$ obtained by solving the Loewner differential equation with $W_t$ as driving term. The Loewner chain is generated by an SLE$_2$ path in
$\Half$ that we denote by $\gamma(t)$.  Let $\hat \gamma(t)
=  F^{-1} \circ \gamma(t)$ which is an SLE$_2$ path
from $a$ to $ b$ in $D_A$
parametrized by capacity in $\Half$. (This parametrization depends
on $F$ but we have fixed $F$.)   
We write
\[
F^{\SLE}_n(z) = (g_{\tau_n}^\SLE \circ F)(z)-W_{\tau_n}
\]
and
\[
F^{\LERW}_n(z) = (g_n \circ F)(z) - U_n.
\]
  
   \begin{lemma}\label{lem:coupling-of-maps}
  There is $c < \infty$ such that the following holds. For $R$ sufficiently large, except on an event of probability at most $ch^{1/10}$, we have  uniformly in $\zeta \in A$ such that $\Im F_n^\SLE(\zeta) \ge h^{1/80}$, 
  \[
  \left| F_n^{\LERW}(\zeta) -  F_n^{\SLE}(\zeta)\right|  \le c  h^{1/15}.
  \]
  Moreover, if $y \ge h^{1/80}$ and $f_n^{\LERW} = g_n^{-1}, \, f_{\tau_n}^{\SLE} = (g_{\tau_n}^{\SLE})^{-1}$, then
  \[
  \left|f_n^{\LERW}(z) - f_{\tau_n}^{\SLE}(z) \right| \le c h^{1/15}
  \]
  and
 \[
 \left|y|(f_n^{\LERW})'(z)| - y|(f_{\tau_n}^{\SLE})'(z)| \right| \le c h^{1/15}.
 \]
  \end{lemma}
  \begin{proof}
  This follows from Proposition~\ref{prop:main-coupling} using Proposition~\ref{prop:loewner-comparison} and Proposition~\ref{prop:reverse-time-comparison} with  the choices
    \[   \ee = h^{1/10}, \quad \delta = h^{1/80}.
  \]
  \end{proof}

\section{Coupling the paths}
We continue the study of the coupling of Proposition~\ref{prop:main-coupling} and Lemma~\ref{lem:coupling-of-maps}. We will show that with large probability, the LERW and the SLE$_2$ curves stay close in the supremum norm when they are paramatetrized by capacity. In the process we will also derive an estimate on the maximal diameter increments of both paths. 
We will work in the same set-up as the previous section  and first consider $(A,a,b,F)$ with the conformal transformation $F$ fixed.
We give some additional notation. In this section it will be convenient to slightly abuse notation and drop the $m_n = m_n(h)$  and simply write
  \[
  \eta^n =\eta^{m_n}, \quad \eta^n_j = \eta^{m_n}_j, \, j=1,2,\ldots, m_n,
  \]
  for the SAW given by the first $n$ mesoscopic capacity increments of the LERW in $A$. We write \[
  \eta_{{{\smallsquare}}}^n = \bigcup_{x \in \eta^n}\Square_x , \quad K_n, \quad K^n = g_{n-1}(K_n),
  \] for  the ``thickened'' LERW in $D_A$, the corresponding LERW hulls in $\mathbb{H}$, and the $n$th mesoscopic hull increment in $\Half$, respectively. Note that $K_n$ is the hull generated by $F(\eta_{\smallsquare}^n)$, i.e., $K_n$ consists of $F(\eta_{\smallsquare}^n)$ together with points disconnected from $\infty$ by the set $F(\eta_{\smallsquare}^n)$. Even though $\eta$ is a simple curve, $\eta_{\smallsquare}$ may generate a strictly larger hull. The domain $D_n$ is formed by cutting out $F^{-1}(K_n)$ from $D_0 = D_A$. We write
  \[
  a_n = [\eta^n_{l-1} ,\eta^{n}_l], \quad l = |\eta^n|,
  \]
  for the tip edge of $\eta^n$; as usual we identify an edge with its midpoint. 
Up to this point we have only compared the images of the LERW and SLE$_2$ in $\Half$. Let us now fix an analytic domain $D$ containing $0$ as an interior point and with $a',b' \in \partial D$ fixed, and then take $A=A(N,D)$ to be the lattice domain which approximates $N \cdot D$ in the sense that $D_A$ is the largest union of squares domain contained inside $N \cdot D$. We think of $N^{-1}$ as the mesh size. We choose $a, b \in \partial_e A$ among the edges closest to $Na',Nb'$ respectively and for each $N$ we fix one choice of $F:(D_{A},a,b) \to (\Half, 0, \infty)$. We shall require later that $R=R_{A,a,b, F}=4|(F^{-1})'(2i)|$ is sufficiently large compared to $N$. Since $D$ is analytic, there is a constant $c$ depending only on $D$ such that $|a - N a'| + |b - N b'| \le c \log N$ (see Section~7 of \cite{LV_LERW_natural}).

We write \[\check D =  N^{-1} D_A\]  for the scaled domain which approximates $D$ from the inside and we set $\check{a} = N^{-1}a$ and $\check{b} = N^{-1} b$. Note that 
\[
|\check{a} - a'| + |\check{b} - b'| \le c \frac{\log N}{N}.
\]
 Given the LERW in $A$ and the associated SLE$_2$  we then have corresponding scaled quantities living in $\check D$:
  
  \[
  \check{\eta}, \quad \check{a}_n, \quad \check{\eta}_{\smallsquare}^{n}, \quad  \check{D}_n, \quad  \check{\gamma}(t),
  \] 
   defined by
  \[
  N^{-1}\eta, \quad N^{-1} a_n, \quad N^{-1}\eta_{\smallsquare}^{n}, \quad N^{-1} D_n , \quad \check{F}^{-1}({\gamma}(t)),
  \]
  respectively,
   where we define \[\check{F}(z)=F(Nz): \check{D} \to \Half, \quad \check{F}(\check a) = 0, \quad \check{F}(\check b) = \infty .\] 
   Capacity is measured using $F$ (or $\check{F}$) as before. We have fixed $F$ and given this it is useful to define a ``nearby'' map from $D$. Define \[\phi(z): D \to \mathbb{H}, \quad \phi(a')=0, \quad \phi(b') = \infty\] by
   \[
   \phi = \check{F} \circ  \tilde{\psi},
   \] 
   where $\tilde \psi$ is chosen so that    
\[\tilde{\psi}: D \to \check D, \quad \tilde{\psi}( a') =  \check{a}, \quad \tilde{\psi}( b') =  \check{b}\] and $\sup_{z \in D}|\tilde{\psi}(z) - z| \le c R^{-1}\log R$ (see Section~7 of \cite{LV_LERW_natural}).  This uses both the regularity of $D$ and that we are approximating by a ``nice'' union of squares domain $\check{D}$.
%
%
  Let \[V = \check{F}^{-1}(\ball(0, 10)).\]  Then there exists $c$ depending only on $(D,a',b')$ such that $\forall z \in V$,
   \begin{equation}\label{map-comparison}
   |\phi(z) - \check{F}(z)| \le c \frac{\log R}{R}|\phi'(z)|, \quad \frac 1 c \le |\phi'(z)| \le c \frac N R.
  \end{equation}
  Consequently, there is a constant $c'$ such that if $U \subset \ball(0,10)$ is a connected set of diameter bounded by $r$,  then \[\diam(\check{F}^{-1}(U)) \le c'\max\{ r, (\log R)/R)\}.\]

Note that the traces of the (stopped) LERW $\check{\eta}^{n_0}$ and SLE $\check{\gamma}[0,\tau_{n_0}]$ in $\check{D}$ are contained in $V$ by the definition of $n_0$. We will implicitly use this fact several times below.

  \subsection{Regularity estimates}
We need some standard estimates on the derivative of the SLE transformation which will give some regularity information for LERW. The first lemma is an easy consequence of  Theorem~4.1 of \cite{JVL}. 
In order to state the lemma we define the following parameters.
\[\beta_+=\frac{2(\sqrt{10}-1)}{9} > .48, \quad q(\beta) = -1 + 2\beta + \frac{\beta^2}{4(1+\beta)}.\]
Notice that $q(\beta) > 0$ if $\beta > \beta_+$. If we choose $\beta = .65  > \beta_+$, then $q(\beta)> 1/3$  and $1-\beta > 1/3$. For $0 < y \le 1$, define the random variable
\[
M(y) = \sup_{t\in [0,1]} y|(f_t^\SLE)'(W_t+iy)|, \quad f^{\SLE}_{t} := (g_t^{\SLE})^{-1}.
\]
\begin{lemma}\label{lem:derivative_estimate}
Suppose $\beta > \beta_+ $ and $q < q(\beta)$. There is a constant $c < \infty$ such that the following holds for each $0 < y_0 <1$,
\begin{equation}\label{oct1.12}
\PP \left\{ \sup_{y\in (0, y_0]} y^{\beta-1}M(y) > c \right\} =O(y_0^q).
\end{equation}
\end{lemma}
\begin{proof}
See Appendix~A of \cite{JV} and set $\kappa = 2$.
\end{proof}
By arguing as in Section~3 of \cite{JVL}, Lemma~\ref{lem:derivative_estimate} directly implies a uniform Holder continuity estimate for $\gamma(t)$, the SLE path in $\mathbb{H}$. For $0 < y < 1$, define
\[
M_{\gamma}(y) = \sup_{t \in [0,1-y^2]} \sup_{s \in [0, y^2]} \left|\gamma(t+s)-\gamma(t) \right|.
\]
\begin{lemma}\label{lem:moc_estimate}
Suppose $\beta > \beta_+$ and $q < q(\beta)$. There is a constant $c < \infty$ such that
\[
\PP\left\{\sup_{y \in (0, y_0]} y^{\beta-1}M_{\gamma}(y) > c \right\} =O(y_0^{q}).
\]
\end{lemma}
\begin{proof}
Fix $\beta_0 > \beta_+$ and $q_0 < q(\beta)$. Choose $\beta_1$ with $\beta_+ < \beta' < \beta_0$ sufficiently close to $\beta_0$ so that $q(\beta_1) > q_0$. Then choose $q_1$ such that $q_0 <q_1 < q(\beta_1)$. We let $c_1$ be such that \eqref{oct1.12} holds with parameters $\beta, q$ replaced by $\beta_1, q_1$. In other words, for $0 < y_0 < 1$, the event
\[
\sup_{y\in (0, y_0]} y^{\beta_1-1}M(y) \le c_1
\]
holds with probability at least $1-c_1y_0^{q_1}$. On this event we can then argue as in \cite{JV}.
\end{proof}

  \begin{lemma}\label{lem:sle-diam-estimate}
  We have
  \[
  \Prob \left\{\max_{n \le n_0}\diam\left(\gamma[\tau_n, \tau_{n-1}] \right) > h^{1/30}  \right\} = O(h^{1/30}).
  \]
  \end{lemma}
  \begin{proof}
  First note that except on an event of probability $O(h^{1/5})$, we have the very rough bound \[\max_{n \le n_0}|\tau_n -\tau_{n-1}| \le ch^{1/5}.\] We apply Lemma~\ref{lem:moc_estimate} with $y_0 = 2h^{1/10}$ and $\beta = .65$ so that $1-\beta > 1/3$ and $q > 1/3.$
  \end{proof}

  We continue with a regularity estimate for LERW. Let $I_n$ be the closure of the smallest interval containing  $g^{n}(\partial K^n)$.
Write $u_n$ for the midpoint of $I_n$ and \[\delta:=h^{1/80}, \quad z_n := f_n^\LERW(u_n + i\delta) .\]
 Note that on the event that $\diam K^n \le h^{2/5}$, we have $|I_n| \le c h^{2/5}$. We can think about the Loewner process $u_{n}$ as a (mesoscopic scale) driving term for the LERW, albeit $u_{n}$ is measurable with respect to the SAW and not the other way around.

   We define $\delta = h^{1/80}$.
  \begin{lemma}\label{lem:dn}
 If \[d_n=\dist (z_n, \partial H_n),\]
  then
  \begin{equation}\label{jan13.1}
  \Prob \left\{ \max_{n \le n_0} d_n > \delta^{1/3} \right\} = o(\delta^{1/3}) .
  \end{equation}
  \end{lemma}
  \begin{proof}
  By distortion estimates $d_n \asymp \delta|(f_n^\LERW)'(u_n + i\delta)|$. 
  Let $E$ be the event that \[\max_{n \le n_0}\delta|(f_{\tau_n}^\SLE)'(W_{\tau_n} + i\delta)| \le \delta^{1/3};\]
  \[\max_{n\le n_0}\left|\delta|(f_n^\LERW)'(u_n+i\delta)| -\delta|(f_{\tau_n}^\SLE)'(u_n + i\delta)|  \right| \le h^{1/15};\]
\[ \max_{n \le n_0}|W_{\tau_n} - u_n| \le h^{2/5}.\]
Using distortion estimates again, we see that $d_n \le c \delta^{1/3}$ holds on $E$. But from Lemma~\ref{lem:derivative_estimate}, Proposition~\ref{prop:reverse-time-comparison} and the fact that $|I_n| \le c h^{2/5}$ in the coupling,  we have $\Prob\left\{E^c \right\} = o(\delta^{1/3})$.  
  \end{proof}
  We would now like to say that the distance between $F(a_{n})$, the image in $\Half$ of the tip of the LERW, and the point $z_n$ is small. This would then allow us to gauge the distance between the LERW tip to the SLE tip by a ``$4\ee$-argument'' using the estimates from Section~\ref{sect:deterministic}. The obvious strategy is to try to estimate the length of the hyperbolic geodesic connecting $F(a_n)$ with $z_n$. For the SLE we can do this by integrating the derivative estimate of Lemma~\ref{lem:derivative_estimate}, but for LERW we do not have this kind of estimate. What we know is that $d_n$ is small. That is, we know that the distance to \emph{some} point on the boundary is small, but we are interested in the distance to a particular point. It is not hard to draw curves for which $d_{n}$ is small but the distance to $z_{n}$ is large -- think of a curve which creates a large bottleneck close to $z_{n}$ (but with $z_{n}$ outside the ``bottle'') and then enters and travels far into the ``bottle''. So for this to work we need to use some regularity property of LERW. The first step is a geometric argument which will allow us to estimate the diameter of the geodesic in terms of $d_n$ on the event that certain crossing events do not occur. The second step is to prove that such crossing events are unlikely for LERW. We also get a H\"older-type estimate for LERW in the capacity parameterization on a coarse enough scale.

A crosscut of a simply connected domain $D$ is a continuous simple curve $\sigma = (\sigma(t), t \in [0,1])$ such that $\sigma(0+), \sigma(1-) \in \partial D$ and $\sigma(t) \in D, \, t\in (0,1)$. 
  \begin{lemma}\label{lemma:crosscut}
  There exists  $1< r_0<\infty$ such that the following holds. Suppose $h < 1/r_0$. For every $n \le n_0$,   $\partial \ball(z_n, r_0d_n)$ contains a crosscut of $H_{n-1}$ that separates $z_n$ and $f_{n-1}^\LERW(K^n)$ from $\infty$ in $H_{n-1}$.
  \end{lemma}
 \begin{proof}
 We write $f_n = f_n^{\LERW}$. We can assume $u_n=0$, so that $I_n$ is centered around $0$. Note that $\hm (i\delta, I_n, \mathbb{H}) \le c h^{3/10}$, where $c$ does not depend on $A,n$ and $\hm$ refers to harmonic measure. We claim that if $h$ satisfies $c h^{3/10} \le 1/4$, then if $J$ is any bounded open interval with $\hm(i\delta, J, \mathbb{H}) > 3/4$ we have that $I_n \subset J$. Indeed, \[\hm (i\delta, I_n, \mathbb{H}) \le 1/4\] implies \[\hm(i \delta, I_n \cup \R_+, \mathbb{H}) =\hm(i \delta, I_n \cup \R_-, \mathbb{H})\le 5/8< 3/4.\] 
 
 For each $r > 1$, let $\Sigma$ be the collection of crosscuts of $H_n$ formed by the elements of $\partial \ball(z_n, r d_n) \cap H_{n}$. The crosscuts in $\Sigma$ have disjoint interiors, and, being crosscuts, they each partition $H_n$ into exactly two components, one bounded and one unbounded (since $\eta^n_{\smallsquare}$ is bounded). Since $\partial \ball(z_n, r d_n) \cap H_{n}$ separates $z_n$ from $\infty$ in $H_n$, $\Sigma$ contains at least one crosscut which separates $z_n$ from $\infty$ in $H_n$. Let $\sigma$ be the unique crosscut in $\Sigma$ which: $a)$ separates $z_n$ from $\infty$ in $H_n$, and $b)$ if $\sigma \neq \sigma' \in \Sigma$ and $\sigma'$ also separates $z_n$ from $\infty$ in $H_n$, then $\sigma$ separates $\sigma'$ from $\infty$ in $H_n$. Write $\Omega$ for the bounded component of $H_n \setminus \sigma$, and $E = \partial \Omega \setminus \sigma$. Then $\Omega$ is a simply connected domain and $z_n \in \Omega$. By the Beurling projection theorem and the maximum principle, we can find a universal $r'< \infty$ such that if $r > r'/2$ then $\hm(z_n, E, \Omega) > 3/4$. Choose $r_0=r'$ and let $\sigma = \sigma_{r_0}$. Then $g_n(\sigma)$ is a (bounded) crosscut of $\mathbb{H}$ which separates $i\delta$ and an interval $J$ from $\infty$.  By conformal invariance and the maximum principle, $\hm(i\delta, J, \mathbb{H}) > 3/4$. Consequently, by the first paragraph of the proof, $I_n \subset J$. From this we conclude that $\sigma$ separates $z_n$ and $f_{n-1}(K^n)$ from $\infty$ in $H_n$. Since $H_{n} = H_{n-1} \setminus f_{n-1}(K^n)$ and $f_{n-1}(K^n)$ is disjoint from $\sigma$, it follows that $\sigma$ also separates $f_{n-1}(K^n)$ and $z_n$ from $\infty$ in $H_{n-1}$.
 \end{proof} 
 We now quote two results from \cite{LV_LERW_natural} that we will use in the proof of the regularity estimate for LERW. Here we write $\whoknows_r$ for triples $(A,a,b)$ where $A$ is a $\ZZ^2$ lattice domain with boundary edges $a,b$ as before but with the added requirement that $A$ contains $C_r = \{z \in \ZZ^2 : |z| < r\}$. We write $I_r$ for the set of self-avoiding walks that include at least one
vertex in $C_r$. The next proposition describes the bottleneck event.
 \begin{prop}\label{prop:lemma-six-arm}
There exist $c < \infty$ such that the following holds.
 Suppose $0 < r < s$ and $(A,a,b) \in \whoknows_r$ with
 $|a_+| < r$. Let \[E''=E''(0,r, s)\] denote the set of  
  $\eta = [\eta_0,\ldots,\eta_n]   \in \saws_A(a,b)
$ such that  there exists $ 0 < j_1
   < k_1 < j_2 < k_2 < n$, with $|\eta_{j_1}|, |\eta_{j_2}| \geq s$ and $|\eta_{k_1}|, |\eta_{k_2}| \le r$.  Then
 \[  \Prob_{A,a,b}\left\{ E'' \right\} \leq   
 c \, (r/s)^2  .\]
 \end{prop}
 \begin{proof}
 See Section~6 of \cite{LV_LERW_natural}.
 \end{proof}
 \begin{rem}\label{rem:four-arm}
 Section~6 of \cite{LV_LERW_natural} also proves a similar result where the event $E''$ is replaced by the event $E' = E'(0,r,s)$ that the path starts at distance $r$ from $0$, gets away to distance $s$ and then returns to distance $r$ again. The bound in this case is  $O(r/s)$. In fact, the proofs of these two estimates are nearly the same and are both contained in Proposition~6.16 of \cite{LV_LERW_natural} which estimates the corresponding probability for random walk conditioned to exit at $b$. We do not expect the exponents to be sharp. 
 \end{rem}
  \begin{prop}  \label{corollary.bastille1}
If $(A,a,b) \in \whoknows_{2r}$, then
\begin{equation}  \label{jul14.3}
  \Prob_{A,a,b}\{0 \in \eta \mid \eta \in I_r\} \asymp
   r^{-3/4}, 
   \end{equation}
and hence
\begin{equation}  \label{jul14.4}
       \Prob_{A,a,b}[I_r]
    \asymp r^{3/4} \,   \Prob_{A,a,b}\{0 \in \eta  \}.
    \end{equation}
\end{prop}
\begin{proof}
See Section~6 of \cite{LV_LERW_natural}
\end{proof}
 
\begin{prop}\label{prop:diameter-bound}
There are constants $c_1,c_2, c_3$ such that if $r = c_1 \delta^{1/3}$, then    
  \[
  \Prob \left\{ \max_{n \le n_0}\left[ \diam \left( \check{\eta}_{\smallsquare}^n \setminus \check{\eta}_{\smallsquare}^{n-1} \right)  \right]> r^{1/8} \right\}  \le c_2 r^{1/6} + c_3 r (N/R)^2
  \]
  and
  \begin{equation}\label{distdnzn}
   \Prob \left\{ \max_{n \le n_0} |\check a_n-\check{z}_n|> r^{1/8}  \right\}  \le c_2 r^{1/6} +  c_3 r (N/R)^2.
  \end{equation}
    \end{prop} 
  \begin{proof}
  The idea is to use Lemma~\ref{lemma:crosscut} to see that the event that the diameter of $ \check{\eta}_{\smallsquare}^n \setminus \check{\eta}_{\smallsquare}^{n-1}$ is large implies the existence of a bottleneck event for the LERW of the type considered in Proposition~\ref{prop:lemma-six-arm}. By covering the domain with annuli we can then show that the probability that such a crossing occurs is small.

   We will use notation from the proof of Lemma~\ref{lemma:crosscut}. We will write \[\Delta_{n} = \check \eta_{\smallsquare}^n \setminus  \check \eta_{\smallsquare}^{n-1}\] for the $n$:th increment of the LERW in $\check{D}$.
By \eqref{jan13.1} there is a constant $c < \infty$ such that except on an event of probability $O(\delta^{1/3})$, $\sup_{n}d_n \le c \delta^{1/3}$. Therefore, by Lemma~\ref{lemma:crosscut}, on this event (which we assume to be on from now on), there is an $1 < r_0 < \infty$ and for every $n \le n_0$  a crosscut  $\sigma \subset \partial \ball(z_n, r_0d_n)$ of diameter $O(\delta^{1/3})$ which separates $F(\Delta_{n})$ from $\infty$ in $H_{n-1}$. Let us fix an arbitrary $n$ -- the estimates will not depend on the choice. Let $\check{\sigma} = \check{F}^{-1}(\sigma)$. Then since $ \log R/R = o(\delta)$, there is a constant $c'$ (which depends only on $D$) such that $\check{\sigma} \subset \ball(\check{z}_n, c'\delta^{1/3}).$  Define \[r:=c' \delta^{1/3}.\] 
Let us now suppose that $\diam \left(\Delta_{n}\right) \ge r^{1/8}$. We will argue that this forces a bottleneck event of the type descibed in Proposition~\ref{prop:lemma-six-arm} to occur.

Suppose first that $\dist(\check{z}, \partial \check{D}) \ge 4r$. Then $\check{\sigma}$ does not intersect $\partial \check{D}$  since it is contained in the ball of radius $r$ about $\check{z}_{n}$. By modifying $\check \sigma$ slightly we can assume that $\check \sigma$ touches $\check \eta^n_{\smallsquare}$ in exactly two distinct squares. 
Let \[j^* = \min\{ j \ge 0 : \check{\eta}_{\smallsquare}^n(j) \, \cap \, \check \sigma \neq \emptyset\}, \quad k^* = \max\{ j \ge 0 : \check{\eta}_{\smallsquare}^n(j) \,  \cap \,  \check \sigma \neq \emptyset\}.\] We know that $\check \sigma $ separates $\Delta_n$ from $\check b$.  Since $\check{\sigma}$ does not touch $\Delta_{n}$ nor $\partial{\check{D}}$ the assumption $\diam \Delta_n \ge r^{1/8}$ implies that we must have $$\diam \check{\eta}_{\smallsquare}[j^*, k^*] \ge r^{1/8}-r \ge r^{1/8}/2$$ if $r$ is small enough. We have $\Delta_n \subset \check{\eta}_{\smallsquare}[k^*+1, m_n]$ and hence  
  $$\diam \check{\eta}_{\smallsquare}[k^*+1, m_n] \ge r^{1/8}.$$ (Here we remark that \[| \check a_n - \check{z}_n| \le 2r+ \diam \check{\eta}_{\smallsquare}[k^*+1, m_n] ,\] which will give \eqref{distdnzn}.)  Another way to phrase this is as follows: $\eta^n$ gets to distance $r$ of $\check z_n$ (at time $j^*$), then gets away to distance at least $r^{1/8}/2$, returns to distance $r$ (before time $k^*$), then gets away to distance at least $r^{1/8}/2$ a second time in order for $\diam \Delta_{n} \ge r^{1/8}$. This is not yet a bottleneck event of type $E''$, but since the path continues to $\check b$ and $\check \sigma$ separates the tip from $\check b$, the path must return to distance $r$ from $\check{z}_{n}$ a third time. The conclusion is that the assumption that $\diam \Delta_n \ge r^{1/8}$ implies that the bottleneck event $E''(\check{z}_n, r, r^{1/8}/2)$ occurs in the case when $\dist(\check{z}_{n}, \partial \check{D}) \ge 4r$.
  We now estimate the probability that there occurs a bottleneck event as just described.  For this it will be useful to consider two cases. Let \[U = \left\{z \in \Half: r^{1/3} \le |z| \le 10 \right\}, \quad \check{U} = \check{F}^{-1}(U).\]
We first estimate the event that there exists $w \in \check{U}$ such that $1_{E''(w,r,r^{1/8}/2)} = 1$. If such a $w$ exists there exists a point $w'$ on the grid $(r/2)\, \ZZ^{2}$ such that $1_{E''(w',2r,r^{1/8}/3)} = 1$. Using Proposition~\ref{corollary.bastille1} we can see that the probability that $\check{\eta}$ gets to distance $2r$ from a fixed point $w' \in \check U$ is $O(r^{1/2})$. Given this, by Proposition~\ref{prop:lemma-six-arm}, the probability that $E''(w', 2r, r^{1/8}/3)$ occurs is $O((r/r^{1/8})^2) = O(r^{7/4})$. We can then cover $\check{U} $ by $O(r^{-2})$ balls and apply these bounds for the center of each ball to see that except on an event of probability  $O(r^{1/4})$ there will be no point $w \in \check{U}$ such that $1_{E''(w,r,r^{1/8}/2)} = 1$. For points whose image in $\Half$ is at distance $r^{1/3}$ from $0$ we estimate the probability of getting to distance $2r$ by $1$ and then proceed as above. The resulting bound on the probability of a bottleneck is of order $r^{-2 + 2/3 + 7/4}=r^{5/12}$.

It remains to consider the case when $\dist(\check{z}_{n}, \partial \check{D}) < 4r$.  On this event we can use the uniform derivative estimate (Lemma~\ref{lem:derivative_estimate}) for the SLE to see that the tip of the SLE coupled to the LERW at time $\tau_{n}$ (which is near $z_{n}$) is at distance $O(rN/R)$ from $\partial \Half$. Therefore, since the boundary exponent for SLE$_{2}$ equals $3$, the probability that such an event occurs for $z_{n} \in U$ is of order $r(N/R)^2$. Finally, for the remaining case when $z_{n}$ is not in $U$ but near $\partial \Half$ the image of the LERW must get to distance $r^{1/8}$ from $0$ and return to distance $r^{1/3}$. The probability of this event can be estimated by the probability of the same event for random walk excursion. This way we get a bound of order $r^{1/3-1/8}=r^{5/24}$. This completes the proof.

  \end{proof}

  \begin{cor}\label{cor:coupling}
  There exist $c_1,c_2,c_3$ such that if  $r$ is as in Proposition~\ref{prop:diameter-bound}, then
  \[
  \Prob \left\{ \max_{n \le n_0}\left\{ \left| \check{\gamma}(\tau_n) - \check{a}_n \right|  \right\} > r^{1/8} \right\} \le c_2r^{1/6} + c_3 r(N/R)^2.
  \]
and
  \[
  \Prob \left\{\sup_{0 \le t \le 1}\left|\check\gamma(t) - \check\eta(t) \right| > r^{1/8}  \right\} \le c_2r^{1/6} + c_3 r(N/R)^2.
  \]
    \end{cor}
 \begin{proof}Using \eqref{map-comparison}, Lemma~\ref{lem:dn} and Proposition~\ref{prop:diameter-bound}, we see that there is an event $E$ with $\Prob\left\{E^c \right\} = O(r^{1/6} + rN/R)$ on which we have the estimates
 
 \begin{align*}
 \left| \check{\gamma}(\tau_n) - \check{a}_n \right| &\le c\left| \check \gamma(\tau_n) - \check{F}^{-1} \circ f_{\tau_n}^\SLE(W_{\tau_n} + i\delta) \right| + \left| \check a_n  -\check{F}^{-1} \circ f_{\tau_n}^\SLE(W_{\tau_n} + i\delta) \right| \\
& \le  O(\delta^{1/3}) + \left| \check a_n -\check z_n \right| +  \left|\check z_n - \check{F}^{-1} \circ f_{\tau_n}^\SLE(W_{\tau_n} + i\delta) \right| \\
& \le O(r^{1/8} ) + c\left|z_n - f_{\tau_n}^\SLE(W_{\tau_n} + i\delta) \right| \\
& \le O(r^{1/8})  + c\left|z_n - f_{\tau_n}^\SLE (u_n + i\delta)  \right| + c\left|f_{\tau_n}^\SLE(u_n + i\delta) - f_{\tau_n}^\SLE (W_{\tau_n} + i\delta) \right| \\
&=O(r^{1/8} ).
 \end{align*} 
%
 This gives the first assertion. By Lemma~\ref{lemma:large-cap} the event that $\hcap \eta^{n_0} \ge 2$ has probability $1-o(r^{1/6})$. Since we know that $\diam \eta^{n_0} \le 2$ the Beurling estimate implies $\hcap \eta^n \le \hcap \eta^n_{\smallsquare} \le \hcap \eta^n + O(R^{-1/2})$. The second assertion follows.
\end{proof} 
Recalling that we have defined $h=R^{-2u/3}$, $\delta = h^{1/80}$ and $r=c_1\delta^{1/3}$, we get the following corollary. 
\begin{cor}
There exists $p_{0} > 0$ and for every $p \in (p_0, 1]$ a $q>0$ such that the following holds. Let $(D,a',b')$ be given with $D$ analytic and $a',b' \in \partial D$ and for $N \ge 1$, let $(A, a, b)$ be as in the beginning of the section, approximating $N \cdot D$. If $R=R_{A ,a,b,F} \ge  N^{p}$ for $N$ sufficiently large, then
  \[
  \Prob \left\{\sup_{0 \le t \le 1}\left|\check\gamma(t) - \check\eta(t) \right| > R^{-q} \right\} < R^{-q}.
  \]
\end{cor}

 \section{Proof of Lemma \ref{jan26.lemma1}}  \label{newproofsec}
 
\begin{lemma}  There exists $c>0 $ such that the following
holds.   Let $\sigma_r$
be the first index  $j$
such that $ \Im[F(\eta_j)] \geq 2r$.
 Then for $R^{-1/4} \leq r  \leq c$,
\[  \Prob_{A,a,b}\{ -r \leq
 \Re[\eta_j]   \leq r \mbox{ for all }
  j \leq \sigma_r\}  
 \geq c.\]
\end{lemma}

We note that $\hcap\left(\eta[0,\sigma_r]\right)
 \geq r^{2}$.

\begin{proof}Let $\tilde{\omega}$ denote the excursion
so that $\eta = \text{LE}[\tilde{\omega}]$, and for ease
of notation let use write $ \omega_k
 = F[\tilde{\omega}_k]$. 

We first consider the following event for the random walk
excursion.  Let $\rho$ be the first
$j$ with $\Im[\omega_j] \geq 4r$ and consider
the event that
\[  - r \leq  \Re[\omega_j]\leq r, \;\;\;
  0 \leq j \leq \rho, \]
\[     \Im[\omega_j]  \geq 3r, \;\;\;\;
    \rho \leq j < \infty.\]
Note that on this event,  if $\eta$ is the loop-erasure of $\omega$, then 
\[  -r \leq
 \Re[\eta_j]   \leq r , \;\;\; 0 \leq
  j \leq \sigma_r.\]
Hence, we need to show that this event on excursions
 has
positive probability.  The hard work was done in
\cite[Proposition 3.14]{KL}  
where it is shown that there exists $c'$ such that
with positive probability, if $\rho$ is the first
time $j$  that the excursion reaches $\{\Im(z) \geq c'r\}$,
then $\max\{ |\Re(\omega_j)| : 0 \le j \le \rho\}  \leq r/2$.  (That paper considers the
 map to the unit disk rather than the upper half plane,
 but the result can easily be adapted by mapping the
 disk to the half plane.)  Given this event, the
 remainder of the path can be extended using the
 invariance principle.  Indeed, this follows from the following facts about the Poisson kernel.
 Let us consider
 \[  V = V(A,h) = 
  \{\zeta \in A:  F(\zeta) \in \{|z| \leq 5r\}.\]
  \[  V_- =   V_-(A,r) = \{\zeta \in V:
    \Im[F(\zeta)] \leq r \}, \]
  \[   V_+ = V_+(A,r) = \{\zeta \in V:
      \Im[F(\zeta)] \geq 2r \}. \]
Then by combining (1) and (41) of \cite{KL} ,
we can see that  for $R$ sufficiently large and $R^{-1/4} \leq r
\leq R^{-\epsilon}$, we have for all $\zeta_+ \in V_+,
\zeta_- \in V_-$,
\begin{equation}  \label{jan25.1}
   H_A(\zeta_+,b) \geq \frac 32 \, H_A(\zeta_-,b).
   \end{equation}
   In fact, one can show that there is $u>0$ such that
   \[
   \frac{H_A(\zeta_+, b)}{H_A(\zeta_-,b)} = \frac{\Im F(\zeta_+)}{\Im F(\zeta_-)} \left(1+O(R^{-u}) \right) 
   \]
   so, allowing for the small error, the quotient is at least $3/2$. This estimate implies that the probability that
  an excursion starting at $\zeta \in V_+$ with probability at least $1/3$ does not visit $V_-$.
 
 \end{proof}

We now complete the proof of Lemma \ref{jan26.lemma1}.
Let $\xi_1$ be the first $j$ such that
$|F(\eta_j)| \geq 4r$. Using the Beurling estimate, we have $|F(\eta_j)| \leq 4r + O(R^{-1/2})
 \leq 5r$. Let $F_1 = g_1 \circ F$
where $g_1: F(D_{A_{\xi_1}}) \rightarrow \Half$
with  $g(a_1) = 0$ and $g_1(z) \sim z$ as $ z\rightarrow \infty$.
Inductively, we define $\xi_k$ to be the
first $j=j_k$ such that $|F_{k-1}(\eta_j)| \geq 4r$,
and define $F_k$ in the same way.
Let $J$ be the first $k$ such that
\[  \Im\left[F_{k-1}(\eta_{j_k})\right] \geq 2r.\]
Using the previous lemma, we see that
\[   \Prob\{J \geq k \} \leq e^{-\alpha k}, \]
for some $\alpha > 0$.  In particular, for $R$ sufficiently
large,
\[    \Prob\{J \geq r^{-1/15}\}
   \leq \exp \{-\alpha \lfloor r^{-1/15}\rfloor \} \leq \exp\{
     r^{-1/20}\}.\]

Note that 
$ \hcap[F(\eta_{\xi_J})] \geq
    \hcap\left[F_{J - 1}(\eta_{J})\right] \geq  r^2.$
 We also claim that there exists a universal $c_1 < \infty$
 such that
 \[  \diam\left[F(\eta[0,\xi_J])\right]
    \leq c_1 Jr  .\]
 This is a fact about the Loewner equation.  More
 generally, suppose that $K_1 \subset K_2 \subset
 \cdots$
 is an increasing sequence of connected hulls in $\Half$ with
 corresponding maps $g_{j}: \Half \setminus K_j
 \rightarrow \Half$.  Suppose also that for each
 $j$, $g_{j-1}(K_j \setminus K_{j-1})$ is connected.
 For any connected hull $K$ (see \cite[(3.14)]{Lbook}) we compare the diameter with the (potential theoretic) capacity:
 \[  \diam(K) \asymp \text{cap}_\Half(K)
   := \lim_{y \rightarrow \infty}
     y\, \Prob^{iy}  \{B_T \in K\}, \]
  where $B$ is a complex Brownian motion and
   \[  T = T_K = \inf\{t: B_t \in K \cup \R\}.\]
  If $T_j = T_{K_j}$ with $T_0 = T_\eset$, then
  \[ \Prob^{iy}(K_k)
        = \Prob^{iy}\{T_k < T_0\}
          \leq \sum_{j=1}^k \Prob^{iy}\{T_j < T_{j-1}\}.\]
  Using conformal invariance of Brownian motion and the
  fact that $g_{j-1}(iy) = iy +O(1)$, we can see that
\begin{eqnarray*}
 \lim_{y \rightarrow \infty}y\,\Prob^{iy}\{T_j < T_{j-1}\}
    & = & \lim_{y \rightarrow \infty}y\,\Prob^{g_{j-1}(iy)}
        \{B(T_{g_{j-1}(K_j \setminus K_{j-1})})
         \not\in \R\}\\
         & = & {\rm cap}_\Half[g_{j-1}(K_j \setminus K_{j-1})], 
  \end{eqnarray*}
  and hence,
\begin{align*}
 \diam(  K_k)  \leq  c\, {\rm cap}_\Half(  K_k) 
    & \leq c\, \sum_{j=1}^k  {\rm cap}_\Half[g_{j-1}
    (K_j \setminus K_{j-1})]\\
    & \leq  c \, 
  \sum_{j=1}^k\diam\left[g_{j-1}
    (K_j \setminus K_{j-1})\right].
    \end{align*}  
This concludes the proof.

\end{document}